\documentclass[11pt]{amsart}
\usepackage[T1]{fontenc}
\usepackage{lmodern}
\usepackage{microtype}
\usepackage{amssymb,mathtools,bm}
\usepackage{booktabs,array,enumitem}
\usepackage[dvipsnames]{xcolor}
\usepackage[colorlinks=true,linkcolor=MidnightBlue,citecolor=BrickRed,
urlcolor=RoyalBlue]{hyperref}

\usepackage[nameinlink,noabbrev]{cleveref}

\allowdisplaybreaks
\numberwithin{equation}{section}

\newtheorem{theorem}{Theorem}[section]
\newtheorem{proposition}[theorem]{Proposition}
\newtheorem{lemma}[theorem]{Lemma}
\newtheorem{corollary}[theorem]{Corollary}
\theoremstyle{definition}

\theoremstyle{remark}

\newcommand{\R}{\mathbb R}
\newcommand{\Pp}{\mathbb P}
\newcommand{\Om}{\Omega}
\newcommand{\grad}{\nabla}
\newcommand{\dd}{\,\mathrm d}
\newcommand{\ip}[2]{\langle #1,#2\rangle}
\newcommand{\norm}[1]{\lVert #1\rVert}
\newcommand{\abs}[1]{\lvert #1\rvert}

\newcommand{\esssup}{\operatorname*{ess\,sup}}

\title[$W^{1,\infty}$ Stability of Spectral Ritz Projections]
{Nearly Sharp $W^{1,\infty}$ Stability of Polynomial Spectral Ritz Projections}

\author{Yongxing Guo and Yuwen Li}

\address{School of Mathematical Sciences, Zhejiang University,
	866 Yuhangtang Road, Hangzhou, Zhejiang 310058,
	People's Republic of China}

\email{guoyongxing@zju.edu.cn}
\email{liyuwen@zju.edu.cn}

\thanks{This work was supported by the National Natural Science Foundation of China under grant 12471346.}


\begin{document}
\raggedbottom
	
	\begin{abstract}
		We establish nearly sharp $W^{1,\infty}$ stability estimates for the Legendre spectral Ritz projection associated with Laplacian on a square domain. Its norm grows proportionally to the square root of the polynomial degree, up to a cubic logarithmic factor.  
		The upper bound uses an integral representation of the discrete gradient projection in terms of one-dimensional semigroup kernels and a rank-two Duhamel comparison between the mass-matrix semigroup and its infinite-dimensional counterpart. The lower bound follows from the endpoint concentration of the Legendre kernel.
		As an application, we derive gradient maximum-norm error estimates for the Poisson and heat equations, with spatial factors optimal up to logarithmic terms.
	\end{abstract}
	
	\maketitle
	
	\section{Introduction}\label{sec:introduction}
	
	Spectral methods approximate solutions of differential equations by global
	trigonometric or algebraic polynomials. For smooth solutions, the approximation error can decay faster than any fixed algebraic power of the polynomial degree, with exponential decay in the analytic case. General accounts, from the
	foundational stability and convergence theory to modern Legendre--Galerkin
	algorithms, can be found in
	\cite{GottliebOrszag1977,BernardiMaday1997,
		CanutoHussainiQuarteroniZang2006,ShenTangWang2011}.  For the Poisson equation, the
	Galerkin solution is the Ritz projection in the Dirichlet-energy inner
	product and is therefore contractive in the Dirichlet-energy norm. Stability in $W^{1,p}$ is a
	different question: energy orthogonality gives no direct control of the
	$L^p$ norm of the gradient when $p\ne2$, and the endpoint
	$W^{1,\infty}$ is especially sensitive to the concentration of spectral
	kernels.  This distinction between an approximation estimate for a smooth
	solution and an operator estimate on the same $W^{1,p}$ space is central to the present paper.

    The relevance of Ritz projection stability extends beyond the Poisson equation. In Galerkin approximations of evolution equations with diffusion, the Ritz projection provides a natural error decomposition in which the diffusion contribution of the projection error vanishes by energy orthogonality. Stability in $W^{1,\infty}$ then controls the projected gradient and is particularly relevant to gradient maximum-norm estimates and the analysis of nonlinear transport terms. The present $W^{1,\infty}$ estimates provide a first step toward gradient maximum-norm analysis of more complex PDE systems, such as the Stokes and Navier–Stokes equations, where pressure coupling and the divergence constraint require additional projection estimates. The heat equation illustrates a direct application.

	Maximum-norm error analysis has a well-developed counterpart for finite
	element methods.  For elliptic problems, Scott \cite{Scott1976} derived
	$L^\infty$ error estimates on quasi-uniform triangular meshes, with optimal
	orders for higher-degree elements and a logarithmic factor for piecewise
	linears.  For the Dirichlet Poisson problem on convex polygons, Rannacher
	and Scott \cite{RannacherScott1982} proved that the piecewise linear Ritz
	projection on shape-regular, quasi-uniform meshes is uniformly bounded in
	$W^{1,\infty}$ as the mesh size $h$ tends to zero.  This gives a
	gradient maximum-norm best-approximation estimate with a constant independent
	of $h$, and hence an $O(h)$ gradient error for solutions in
	$W^{2,\infty}$.  The distinction between the function and gradient norms
	matters: this gradient stability does not remove the logarithmic factor
	from the general piecewise linear $L^\infty$ error estimate.
	
	The comparison here is with finite element refinement at fixed local
	polynomial degree.  Its maximum-norm analysis exploits local polynomial
	approximation and weighted estimates for regularized Green functions
	\cite{Scott1976,RannacherScott1982}.  In the spectral method considered
	here, accuracy is increased by raising the degree of a global polynomial
	space on a fixed square.  Local cutoffs and elementwise interpolation do
	not preserve this space, and the constants in fixed-degree finite element
	estimates do not control the limit $N\to\infty$.  More fundamentally, our
	lower bound shows that the spectral Ritz norm on $W^{1,\infty}$ grows at
	least as $N^{1/2}$, whereas the finite element Ritz norm in the setting
	above stays bounded.  Thus the spectral loss is intrinsic to the
	projection, not merely an artifact of an inverse estimate.  This
	operator growth is compatible with exponential convergence for analytic
	solutions, since their best polynomial approximation error decays
	exponentially.
	
	The periodic Fourier case provides a useful comparison because rectangular Fourier truncation is translation invariant and commutes with differentiation.
	The classical M.~Riesz theorem gives uniform $L^p$ boundedness of the
	partial-sum operators for $1<p<\infty$, whereas the one-dimensional
	$L^\infty$ norm grows like the Lebesgue constant, of order $\log N$
	\cite{Zygmund1959}.  On a $d$-dimensional torus, tensor-product truncation
	therefore has $W^{1,p}$ norm bounded independently of $N$ for $1<p<\infty$ and $W^{1,\infty}$ norm of order at most $(1+\log N)^d$.  Moreover, the
	Fourier modes diagonalize both the periodic Laplacian and its heat
	semigroup.  Consequently, elliptic Ritz stability and parabolic spectral
	error estimates reduce directly to Fourier multiplier and approximation
	estimates. Effective maximum principles were established in \cite{Li2021} to obtain maximum-norm bounds for nonlinear Fourier spectral schemes.
	
	Legendre spectral methods on an interval or a box have a different
	structure.  They are adapted to nonperiodic boundary conditions, but they
	are not translation invariant, differentiation does not commute with the
	orthogonal projection, and the kernels concentrate at the endpoints.
	This difference is already visible in classical harmonic analysis:
	Pollard proved that the unweighted Legendre partial-sum operators are
	uniformly bounded on $L^p(-1,1)$ precisely in the restricted range
	$4/3<p<4$ \cite{Pollard1947}, while their $L^\infty$ Lebesgue constants
	grow algebraically, with order $N^{1/2}$ in the Legendre case
	\cite{Cartwright1983}.  Thus the Fourier and Legendre endpoint theories
	cannot be transferred to one another by a routine change of basis.
	
	The numerical-analysis theory of polynomial spectral methods was developed
	through weighted Sobolev approximation estimates for Legendre and
	Chebyshev expansions
	\cite{MadayQuarteroni1981,CanutoQuarteroni1982,BernardiMaday1989} and
	through the analysis of projectors satisfying boundary conditions
	\cite{Maday1990}.  These results yield high-order or exponential elliptic
	error estimates when the exact solution has corresponding additional
	regularity.  The Babu\v{s}ka--Shen basis subsequently made the
	$H^1_0$ Legendre--Galerkin discretization of second- and fourth-order
	problems sparse and computationally efficient \cite{Shen1994}.  None of
	these approximation results, however, identifies the degree dependence of
	the Ritz operator on the zero-trace Lipschitz space when no derivatives
	beyond the first are available. 
	Combining energy stability with a polynomial inverse inequality yields an $N^2$ upper bound for the $W^{1,\infty}$ stability norm, which does not recover the optimal algebraic exponent $1/2$.

	
	Maximum-norm stability and error estimates are also available for
	semidiscrete parabolic finite element equations
	\cite{SchatzThomeeWahlbin1980}.  Such $L^\infty$ estimates concern
	the solution itself and do not directly supply the $W^{1,\infty}$ estimates
	for the global polynomial spaces studied here.  The spectral literature
	includes error estimates for one-dimensional Galerkin semidiscretizations
	with both smooth and nonsmooth initial data
	\cite{deFrutosMunozSola2001}, as well as optimal $L^2$ estimates and sparse
	algebraic systems for space--time Legendre methods
	\cite{Tang2006,ShenWang2007}.  These spectral results are formulated mainly in
	$L^2$, energy, or higher-regularity approximation norms.  In contrast, a
	uniform-in-time gradient maximum-norm estimate for the spatial Legendre--Galerkin
	method requires both $W^{1,\infty}$ stability of the elliptic Ritz map and
	control of the time derivative of the Ritz projection error.  The latter cannot be handled
	sharply by simply combining an inverse inequality with the standard energy
	estimate.
	
	
	
	Establishing $W^{1,\infty}$ stability estimates for the Legendre Ritz projection presents two main difficulties. First, endpoint concentration leads to an unavoidable algebraic loss in the stability bound. Second, the Ritz projection on a square does not factor into one-dimensional projections, and the associated discrete gradient projection couples the two components through the Legendre mass matrix. Consequently, existing one-dimensional projection estimates alone do not yield the required two-dimensional bound.
	
	We establish $W^{1,\infty}$ stability estimates for the Legendre Ritz projection on a square. Its operator norm grows like $N^{1/2}$ up to logarithmic factors, and the algebraic exponent $1/2$ is optimal. 
	The lower-bound construction combines a one-dimensional test function based on the Legendre endpoint kernel with a fixed polynomial in the transverse variable.
	For the upper bound, we combine an integral representation of the discrete gradient projection in terms of one-dimensional semigroup kernels with a rank-two Duhamel comparison between the mass-matrix semigroup and its infinite-dimensional counterpart. The stability bound yields a gradient maximum-norm best-approximation estimate for the Poisson equation.

	We combine this estimate with a negative-norm estimate for the Ritz projection error and with a Hilbert-space convolution bound. This gives a uniform-in-time error estimate for the spatial Galerkin semidiscretization of the heat equation. The negative-norm estimate provides an $N^{-2}$ gain that cancels the $N^2$ inverse-inequality loss. 
	
	\subsection{Main Results}\label{subsec:main-results}
	
	Let $I=(-1,1)$, $\Om=I^2$, and $N\ge2$.  Write $\Pp_N(I)$ for the
	polynomials of degree at most $N$, and define
	\[
	X_N=\Pp_N(I)\cap H^1_0(I),\quad V_N=X_N\otimes X_N,\quad
	\mathcal W=W^{1,\infty}(\Om)\cap H^1_0(\Om).
	\]
	The space $\mathcal W$ consists of Lipschitz functions with zero boundary trace. On this space, Poincar\'e's inequality makes the gradient seminorm equivalent to the full $W^{1,\infty}$ norm.
	For $w\in H^1_0(\Om)$, define $R_Nw\in V_N$ by
	\begin{equation}\label{eq:variational-equation}
		(\grad R_Nw,\grad v_N)_{\Om}=(\grad w,\grad v_N)_{\Om},
		\qquad v_N\in V_N.
	\end{equation}
	The Dirichlet energy is positive definite on $V_N$, so these equations
	have a unique solution. 
	Define the Lipschitz stability constant and the logarithmic factor by
	\begin{equation}\label{eq:Sdef-infty}
		S_{N,\infty}:=\sup_{0\ne w\in\mathcal W}
		\frac{\norm{\grad R_Nw}_{L^\infty(\Om)}}
		{\norm{\grad w}_{L^\infty(\Om)}},
		\qquad \Lambda_N:=1+\log(N+1).
	\end{equation}
	Here and below, vector norms use the Euclidean norm of the components. 
	Throughout this paper, \(A_1\lesssim A_2\) means that \(A_1\le C A_2\) for some generic constant \(C>0\) independent of the degree of polynomials and the quantities under consideration. We write \(A_1\eqsim A_2\) when \(A_1\lesssim A_2\) and \(A_2\lesssim A_1\).

	\begin{theorem}\label{thm:unconditional}
		For every $N\ge2$,
		\begin{equation}\label{eq:endpoint-main}
			N^{1/2}\lesssim S_{N,\infty}\lesssim N^{1/2}\Lambda_N^3.
		\end{equation}
		These bounds also hold for the operator norm of $R_N$ on $W$ equipped with the full $W^{1,\infty}$ norm.
	\end{theorem}

	
	\subsection{Parabolic Application}
	For $p\in\{2,\infty\}$ and $v\in W^{1,p}(\Om)\cap H^1_0(\Om)$, define $E_N^{1,p}(v):=\inf_{v_N\in V_N}\norm{\grad(v-v_N)}_{L^p(\Om)}.$
	Consider $u_t-\Delta u=f$ on $\Om\times(0,T)$ with homogeneous
	Dirichlet boundary conditions and $f\in L^2(0,T;L^2(\Om))$.  If
	$u\in C([0,T];\mathcal W)\cap W^{1,\infty}(0,T;H^1_0(\Om))$,
	its exact semidiscrete Galerkin approximation, initialized by
	$u_N(0)=R_Nu(0)$, satisfies
	\begin{equation}\label{eq:parabolic-intro}
		\begin{aligned}
			\sup_{0\le t\le T}\norm{u(t)-u_N(t)}_{W^{1,\infty}}
			&\lesssim \sqrt N \Lambda_N^3\sup_{0\le t\le T}E_N^{1,\infty}(u(t))\\
			&\quad+\sqrt{\Lambda_N}\,\esssup_{0<t<T}E_N^{1,2}(u_t(t)).
		\end{aligned}
	\end{equation}
	The precise statement and
	proof appear in Theorem~\ref{thm:parabolic-error}. The spatial term carries the factor $N^{1/2}\Lambda_N^3$ from the endpoint Ritz bound. For the term involving $u_t$, the factor is only $\sqrt{\Lambda_N}$. No gradient maximum-norm heat-semigroup
	estimate is assumed. Stationary solutions show that the spatial factor cannot be reduced to $o(N^{1/2})$. The optimal logarithmic factors remain open. Exponential spatial approximation of both
	$u$ and $u_t$, uniformly in time, gives exponential convergence in
	$C([0,T];W^{1,\infty}(\Om))$.
	
	Section~\ref{sec:operator-representation} introduces the basis and the semigroup representation of the discrete gradient projection. Section~\ref{sec:lower-bounds} proves the lower bound, and Section~\ref{sec:endpoint-upper} proves the upper bound through a rank-two Duhamel comparison. Section~\ref{sec:consequences} gives the elliptic approximation estimates and explains their independence of the chosen basis of $V_N$. Section~\ref{sec:parabolic} establishes the parabolic error estimate and proves optimality of its spatial algebraic exponent.
	
	\section{Operator Representation and Factorization}
	\label{sec:operator-representation}
	
	\subsection{The Discrete Gradient Projection}
	
	Define the discrete gradient projection $\mathbb T_NF \in \nabla V_N$, for
	$F\in L^2(\Om;\R^2)$, by
	\begin{equation}\label{eq:Tfull}
		(\mathbb T_NF,\grad v_N)=(F,\grad v_N),\qquad v_N\in V_N.
	\end{equation}
	It is the orthogonal projection onto $\grad V_N$ in $L^2(\Om;\R^2)$.
	For $w\in\mathcal W$, the defining equations give
	\begin{equation}\label{eq:gradient-projection}
		\grad R_Nw=\mathbb T_N\grad w,\qquad
		S_{N,\infty}\le\norm{\mathbb T_N}_{\infty\to\infty}.
	\end{equation}
	The operator on bounded vector fields is well defined because
	$L^\infty(\Om;\R^2)\subset L^2(\Om;\R^2)$.
	Also, $R_N$ is the identity on $V_N$, so $S_{N,\infty}\ge1$.
	The upper bound will follow from an estimate for the discrete gradient projection.
	
	\subsection{Babu\v{s}ka--Shen Basis and the Mass Matrix}
	\label{subsec:basis-mass}
	We use a basis that is orthonormal with respect to the Dirichlet-energy inner product and has a sparse mass matrix. Let $L_k$ denote the Legendre polynomial of degree $k$:
	\begin{equation}
		L_k(x)=\frac{1}{2^k k!}\frac{\dd^k}{\dd x^k}(x^2-1)^k,\qquad k\ge0.
	\end{equation}
	Define the normalized polynomials by
	$\ell_k(x)=\sqrt{(2k+1)/2}\,L_k(x)$ for $k\ge0$.  Then
	$(\ell_k,\ell_j)_{L^2(I)}=\delta_{kj}$.
	
	We introduce the Babu\v{s}ka--Shen functions from \cite{Shen1994} by
	\begin{equation}\label{eq:eta}
		\eta_k(x)=\frac{L_{k-2}(x)-L_k(x)}{\sqrt{4k-2}},\qquad k\ge 2.
	\end{equation}
	The identity $L_k'-L_{k-2}'=(2k-1)L_{k-1}$ gives
	\begin{equation}\label{eq:eta-der}
		\eta_k'=-\ell_{k-1},\qquad
		(\eta_k',\eta_j')=\delta_{kj}.
	\end{equation}
	Thus, for each $N\ge2$, the functions $\{\eta_k\}_{k=2}^{N}$ form an $H_0^1(I)$-orthonormal basis of $X_N$. We collect them into the vector $\eta=(\eta_2,\ldots,\eta_N)^{\top}$. Its mass matrix
	$M=((\eta_k,\eta_j)_{L^2(I)})_{k,j=2}^N$ has the nonzero entries
	\begin{align}
		M_{kk}&=\frac{2}{(2k-3)(2k+1)},\label{eq:Mdiag}\\
		M_{k,k+2}=M_{k+2,k}
		&=-\frac{1}{(2k+1)\sqrt{(2k-1)(2k+3)}}.\label{eq:Moff}
	\end{align}
	In particular, $M$ is positive definite, with nonzero entries only on the main diagonal and the two second off-diagonals. 
	
	The integration operator provides the corresponding representation in the
	Legendre basis. Define the mean-zero subspace of $L^2(I)$ by
	\[
	Y=L^2_0(I):=\left\{h\in L^2(I):\int_I h(x)\dd x=0\right\},
	\]
	and its polynomial subspace by
	\[
	Y_N:=\Pp_{N-1}(I)\cap Y
	=\operatorname{span}\{\ell_1,\ldots,\ell_{N-1}\}.
	\]
	We write $\mathrm{Id}$ for the identity operator on the function space. 	Let $P_N:L^2(I)\to Y_N$ denote the
	$L^2(I)$-orthogonal projection, given by
	\[
	P_Nh=\sum_{n=1}^{N-1}\ip{h}{\ell_n}_{L^2(I)}\ell_n.
	\]
	Define the integration operator $J:Y\to L^2(I)$ by
	\begin{equation}\label{eq:Jdef}
		(Jh)(x)=\int_{-1}^x h(t)\dd t,\qquad h\in Y.
	\end{equation}
	Its range lies in $H^1_0(I)$. Let $J^*:L^2(I)\to Y$ denote the
	adjoint of $J$ with respect to the $L^2(I)$ inner products on both spaces.
	Define the positive operator $K=J^*J$ on $Y$ and its finite-dimensional
	compression $K_N=P_NK|_{Y_N}$. Since
	$J\ell_n=-\eta_{n+1}$, the matrix of $K_N$ in the Legendre basis is
	exactly $M$.
	
	In the full Legendre basis of $Y$, the matrix of $K$ has the same
	sparsity pattern.  Consequently, for $N\ge3$ and $h\in Y$,
	\begin{equation}\label{eq:crossblock}
		\begin{aligned}
			P_NK(\mathrm{Id}-P_N)h
			&=b_N(h,\ell_N)\ell_{N-2}
			+b_{N+1}(h,\ell_{N+1})\ell_{N-1},\\
			|b_N|+|b_{N+1}|&\lesssim N^{-2}.
		\end{aligned}
	\end{equation}
	We shall also use the classical Bernstein estimate for the normalized
	Legendre polynomials \cite{DLMF}:
	\begin{equation}\label{eq:bernstein}
		|\ell_n(\cos\theta)|
		\lesssim\min\{n^{1/2},(\sin\theta)^{-1/2}\},
		\qquad n\ge1,\quad 0<\theta<\pi.
	\end{equation}
	
	\subsection{Semigroup Integral Representation}
	
	We write the discrete gradient projection in block form.
	For $F=(F_1,F_2)^{\top}\in L^2(\Omega;\mathbb R^2)$, we have
	\begin{equation*}
		\mathbb T_N=
		\begin{pmatrix}
			\mathbb T_{11} & \mathbb T_{12}\\
			\mathbb T_{21} & \mathbb T_{22}
		\end{pmatrix},
		\qquad
		\mathbb T_NF =
		\begin{pmatrix}
			\mathbb T_{11}F_1+\mathbb T_{12}F_2\\
			\mathbb T_{21}F_1+\mathbb T_{22}F_2
		\end{pmatrix}. 
	\end{equation*}
	To obtain the separated semigroup kernel representation below, we diagonalize the one-dimensional mass matrix and use its eigenvectors to construct orthonormal systems associated with the Babu\v{s}ka--Shen basis and its derivative basis. 
	
	We collect the Babu\v{s}ka--Shen basis functions and their negative derivatives into the column vectors
	\[
	\boldsymbol\eta=(\eta_2,\ldots,\eta_N)^{\top},
	\qquad
	\boldsymbol d=(d_2,\ldots,d_N)^{\top},
	\qquad d_k=-\eta_k'=\ell_{k-1}.
	\]
	We index the coordinates of $\mathbb R^{N-1}$ by $2,\ldots,N$.
	For $2\le k\le N$, let $e_k$ denote the coordinate vector with entry $1$ at index $k$ and $0$ elsewhere. Then
	\[
	\boldsymbol\eta(x)^{\top}e_k=\eta_k(x),\qquad
	\boldsymbol d(x)^{\top}e_k=\ell_{k-1}(x).
	\]
	Let \((\lambda_j,q_j)\), \(1\le j\le N-1\), be the eigenpairs of the mass matrix \(M\), with the eigenvectors chosen orthonormal.
	\[
	Mq_j=\lambda_jq_j,\qquad q_i^{\top}q_j=\delta_{ij}.
	\]
	
	For \(i,j\in\{1,2\}\), define the integral kernel \(\mathcal K_{ij}\) of \(\mathbb T_{ij}\) by
	\begin{equation}\label{eq:Kij-def}
		(\mathbb T_{ij}f)(x_1,x_2)
		=
		\int_{I^2}\mathcal K_{ij}(x_1,x_2;y_1,y_2)
		f(y_1,y_2)\,\dd y_1\dd y_2.
	\end{equation}
	Define $\chi_i=q_i^{\top}\boldsymbol d$ and 	$\varphi_i= \lambda_i^{-1/2} q_i^{\top}\boldsymbol\eta$. The identities
	\[
	\int_I\boldsymbol d(x)\boldsymbol d(x)^{\top}\,\dd x=I_{N-1},
	\qquad
	\int_I\boldsymbol\eta(x)\boldsymbol\eta(x)^{\top}\,\dd x=M
	\]
	imply that $\{\chi_j\}_{j=1}^{N-1}$ and $\{\varphi_j\}_{j=1}^{N-1}$ are orthonormal in $L^2(I)$.
	Here $I_{N-1}$ denotes the $(N-1)\times(N-1)$ identity matrix.
	Moreover, $\boldsymbol d=-\boldsymbol\eta'$ gives
	$\varphi_i'=-\lambda_i^{-1/2}\chi_i$.
	
	Define the $L^2(I^2)$-orthonormal basis of $V_N$ by
	$v_{ij}(x_1,x_2)=\varphi_i(x_1)\varphi_j(x_2)$,
	$1\le i,j\le N-1$.  Its derivatives satisfy
	\begin{align}
		\partial_1v_{ij}(x_1,x_2)
		&=-\lambda_i^{-1/2}
		\chi_i(x_1)\varphi_j(x_2),
		\label{eq:vij-x}\\
		\partial_2v_{ij}(x_1,x_2)
		&=-\lambda_j^{-1/2}
		\varphi_i(x_1)\chi_j(x_2).
		\label{eq:vij-y}
	\end{align}
	Using the orthonormality of $\{\chi_i\}$ and
	$\{\varphi_i\}$, we obtain
	\begin{align*}
		(\nabla v_{ij},\nabla v_{k\ell})_{L^2(I^2;\mathbb R^2)}
		&=(\partial_1v_{ij},\partial_1v_{k\ell})_{L^2(I^2)}
		+(\partial_2v_{ij},\partial_2v_{k\ell})_{L^2(I^2)}\\
		&=\bigl(\lambda_i^{-1}+\lambda_j^{-1}\bigr)
		\delta_{ik}\delta_{j\ell}.
	\end{align*}
	Thus \(\{\nabla v_{ij}\}_{i,j}\) is an orthogonal basis of
	\(\nabla V_N\).
	
	Throughout the paper, we use the vectors of basis functions and semigroup kernels
	\begin{equation}\label{eq:heat-synthesis}
		\begin{gathered}
			z^1=\boldsymbol d,\qquad z^0=J\boldsymbol d=-\boldsymbol\eta,\\
			H_t^{ab}(x,y)=z^a(x)^{\top}e^{-tM}z^b(y),
			\qquad a,b\in\{0,1\}.
		\end{gathered}
	\end{equation}
	Here $e^{-tM}$ is the semigroup generated by $-M$, and $t\ge0$ is an auxiliary parameter for the elliptic integral representation below.
	In Section~\ref{sec:parabolic}, $t$ denotes physical time, and the discrete heat-equation semigroup is represented in the tensor Babu\v{s}ka--Shen basis by $e^{-tM^{-1}}\otimes e^{-tM^{-1}}$.
	
	\begin{lemma}\label{lem:first-row-kernels}
		The integral kernels of the first-row blocks $\mathbb T_{11}$ and $\mathbb T_{12}$ are given by
		\begin{align}
			\mathcal K_{11}(x_1,x_2;y_1,y_2)
			&=
			\int_0^\infty
			H_t^{11}(x_1,y_1)H_t^{00}(x_2,y_2)\,\dd t,
			\label{eq:audit-K11}\\
			\mathcal K_{12}(x_1,x_2;y_1,y_2)
			&=
			\int_0^\infty
			H_t^{10}(x_1,y_1)H_t^{01}(x_2,y_2)\,\dd t.
			\label{eq:audit-K12}
		\end{align}
	\end{lemma}
	
	\begin{proof}		
		Since \(\mathbb T_N\) is the
		\([L^2(I^2)]^2\)-orthogonal projection onto \(\nabla V_N\), for every
		\(F=(F_1,F_2)\in[L^2(I^2)]^2\),
		\begin{equation}\label{eq:TN-modal-projection}
			\mathbb T_NF = \sum_{i,j}
			\frac{(F,\nabla v_{ij})_{L^2(I^2;\mathbb R^2)}}
			{\lambda_i^{-1}+\lambda_j^{-1}}
			\nabla v_{ij}.
		\end{equation}
		
		To determine \(\mathcal K_{11}\), take \(F=(f,0)\) in
		\eqref{eq:TN-modal-projection} and then take the first component. This
		gives
		\begin{align*}
			(\mathbb T_{11}f)(x_1,x_2)
			&=
			\sum_{i,j}
			\frac{\partial_1v_{ij}(x_1,x_2)}
			{\lambda_i^{-1}+\lambda_j^{-1}}
			\int_{I^2}
			\partial_1v_{ij}(y_1,y_2)f(y_1,y_2)
			\,\dd y_1\dd y_2\\
			&=
			\int_{I^2}
			\left[
			\sum_{i,j}
			\frac{\partial_1v_{ij}(x_1,x_2)
				\partial_1v_{ij}(y_1,y_2)}
			{\lambda_i^{-1}+\lambda_j^{-1}}
			\right]
			f(y_1,y_2)\,\dd y_1\dd y_2.
		\end{align*}
		Using \eqref{eq:Kij-def} and \eqref{eq:vij-x}, we obtain
		\begin{align*}
			\mathcal K_{11}(x_1,x_2;y_1,y_2)
			&=
			\sum_{i,j}
			\frac{\lambda_i^{-1}}
			{\lambda_i^{-1}+\lambda_j^{-1}}
			\chi_i(x_1)\chi_i(y_1)
			\varphi_j(x_2)\varphi_j(y_2)\\
			&=
			\sum_{i,j}
			\frac{\lambda_j}{\lambda_i+\lambda_j}
			\chi_i(x_1)\chi_i(y_1)
			\varphi_j(x_2)\varphi_j(y_2).
		\end{align*}
		The scalar Laplace representations
		\[
		\frac{\lambda_j}{\lambda_i+\lambda_j}
		=\int_0^\infty e^{-t\lambda_i}\lambda_je^{-t\lambda_j}\,\dd t,
		\qquad
		\frac{\sqrt{\lambda_i\lambda_j}}{\lambda_i+\lambda_j}
		=\int_0^\infty \sqrt{\lambda_i}e^{-t\lambda_i}
		\sqrt{\lambda_j}e^{-t\lambda_j}\,\dd t,
		\]
		therefore prove \eqref{eq:audit-K11}.  To determine $\mathcal K_{12}$, take
		$F=(0,f)$ in \eqref{eq:TN-modal-projection} and retain the first
		component.  The same calculation gives
		\[
		\mathcal K_{12}(x_1,x_2;y_1,y_2)
		=\sum_{i,j}
		\frac{\sqrt{\lambda_i\lambda_j}}{\lambda_i+\lambda_j}
		\chi_i(x_1)\varphi_i(y_1)
		\varphi_j(x_2)\chi_j(y_2).
		\]
		Since $z^0=-\boldsymbol\eta$, each mixed semigroup kernel carries a
		minus sign; these two signs cancel in $H_t^{10}H_t^{01}$.
		The second Laplace representation therefore proves \eqref{eq:audit-K12}.
	\end{proof}
	
	The following spectral bounds for $M$ will be used to control the large-time decay of $\mathrm e^{-tM}$.
	
	\begin{lemma}
		\label{lem:mass-spectrum}
		For every \(N\ge2\) and every eigenvalue \(\lambda\in\sigma(M)\),
		\begin{equation}
			\label{eq:audit-spectrum}
			N^{-4}\lesssim\lambda\lesssim 1.
		\end{equation}
	\end{lemma}
	
	\begin{proof}
		Let $(\lambda,q)$ be an eigenpair of $M$, with $Mq=\lambda q$, and
		set $v=q^{\top}\boldsymbol\eta\in X_N$.  The definition of $M$ and the
		$H_0^1(I)$-orthonormality of the Babu\v{s}ka--Shen basis give
		\[
		\|v\|_2^2=q^{\top}Mq=\lambda q^{\top}q,
		\qquad
		\|v'\|_2^2=q^{\top}q.
		\]
		Consequently, $	\lambda=\frac{\|v\|_2^2}{\|v'\|_2^2}.$ 
		Since $v\in X_N\subset H_0^1(I)$, Poincar\'e's inequality gives
		$\|v\|_2\lesssim \|v'\|_2$ and hence $\lambda\lesssim 1$.  On the other
		hand, $v\in\mathbb P_N$ and the polynomial inverse inequality gives
		$\|v'\|_2\lesssim N^2\|v\|_2$.  Thus $\lambda\ge cN^{-4}$, which proves
		\eqref{eq:audit-spectrum}.
	\end{proof}

	\section{The Lipschitz Lower Bound}\label{sec:lower-bounds}
	
	Let $k_N(x,y)=\sum_{j=1}^{N-1}\ell_j(x)\ell_j(y)$ denote the integral kernel of $P_N$, so that
	\begin{equation*}
		P_Nf(x)=\int_I k_N(x,y)f(y)\,\dd y.
	\end{equation*}
	Set $Q=\mathrm{Id}-P_N$. We first prove the one-dimensional lower bound.
	\begin{lemma}\label{lem:epd-endpointlower}
		There exists a mean-zero test function $f_N\in L^\infty(I)$ such that
		\begin{equation}\label{eq:PNh-endpoint}
			(P_Nf_N)(1)\ge c\sqrt N, \qquad \norm{f_N}_\infty\le2.
		\end{equation}
	\end{lemma}
	\begin{proof}
		We first estimate the $L^1(I)$ norm of the endpoint kernel, then choose $f_N$ from its sign. The Christoffel--Darboux formula gives
		\[
		k_N(1,y)+\frac12
		=\frac N2\frac{L_{N-1}(y)-L_N(y)}{1-y}.
		\]
		For $\pi/3\le\theta\le2\pi/3$, the Darboux formula \cite{DLMF}, applied to $L_{N-1}$ and $L_N$, and the cosine difference identity yield
		\[
		\left|
		k_N(1,\cos\theta)+\frac12
		-\frac{\sqrt N\,\sin(N\theta-\pi/4)}
		{\sqrt{2\pi\sin\theta}\,\sin(\theta/2)}
		\right|
		\lesssim N^{-1/2}.
		\]
		On this interval, $\sin\theta$ and $\sin(\theta/2)$ are bounded below by positive constants. Moreover, $|\sin s|$ has period $\pi$ and integral $2$ over each period, so
		\[
		\int_{\pi/3}^{2\pi/3}|\sin(N\theta-\pi/4)|\,\dd\theta
		=\frac23+O(N^{-1}).
		\]
		Changing variables $y=\cos\theta$ therefore gives
		\[
		\begin{aligned}
			\int_I|k_N(1,y)|\,\dd y
			&\ge \int_{\pi/3}^{2\pi/3}|k_N(1,\cos\theta)|\sin\theta\,\dd\theta\\
			&\ge c\sqrt N\int_{\pi/3}^{2\pi/3}|\sin(N\theta-\pi/4)|\,\dd\theta-C\\
			&\ge c\sqrt N-C.
		\end{aligned}
		\]
		For sufficiently large $N$, $C\le(c/2)\sqrt N$, giving a lower bound $(c/2)\sqrt N$. The kernel $k_N(1,\cdot)$ is nonzero for every $N\ge2$, so decreasing $c$ gives
		\[
		\int_I|k_N(1,y)|\,\dd y\ge c\sqrt N,\qquad N\ge2.
		\]
		
		Now define the mean-zero test function by
		\[
		f_N(y):=\operatorname{sgn}k_N(1,y)
		-\frac12\int_I\operatorname{sgn}k_N(1,s)\,\dd s.
		\]
		The subtracted constant is the mean of the first term, hence
		\[
		\int_I f_N(y)\,\dd y=0,\qquad \|f_N\|_\infty\le2.
		\]
		As $P_N1=0$, subtracting this mean leaves the projection unchanged. Thus,
		
		\[
		(P_Nf_N)(1)
		=\int_I k_N(1,y)\operatorname{sgn}k_N(1,y)\,\dd y
		=\int_I|k_N(1,y)|\,\dd y
		\ge c\sqrt N.
		\]
		This proves the lemma.
	\end{proof}
	
	\begin{lemma}\label{prop:linfty-lower}
		For every $N\ge2$, the stability constant satisfies
		\begin{equation}\label{eq:linfty-lower}
			S_{N,\infty}\ge C N^{1/2}.
		\end{equation}
	\end{lemma}
	\begin{proof}
		Let $f_N$ be the mean-zero test function from Lemma~\ref{lem:epd-endpointlower}. 
		Recall that $P_N$ is the $L^2(I)$-orthogonal projection onto $Y_N$, $Q=\mathrm{Id}-P_N$, and $K_N=P_NK|_{Y_N}$. Define
		\[
		g_N=Jf_N,\qquad \psi(y)=\eta_2(y)=\sqrt{3/8}(1-y^2),\qquad w_N(x,y)=g_N(x)\psi(y).
		\]
		The zero mean of $f_N$ gives $g_N(\pm1)=0$, and hence
		\[
		w_N\in\mathcal W,\qquad \norm{\nabla w_N}_\infty\lesssim 1.
		\]
		We shall prove $\sqrt N\lesssim\norm{\nabla R_Nw_N}_\infty+N^{-1}$.
		
		For $\mu\ge0$, define the reaction Ritz projection $r_{N,\mu}g_N\in X_N$ by
		\[
		((r_{N,\mu}g_N)',v')+\mu(r_{N,\mu}g_N,v)
		=(g_N',v')+\mu(g_N,v),\qquad v\in X_N.
		\]
		Set $q=(r_{N,\mu}g_N)'\in Y_N$. Testing with $v=Js$, $s\in Y_N$, and using $r_{N,\mu}g_N=Jq$ gives
		\[
		(q,s)+\mu(Kq,s)=(f_N,s)+\mu(Kf_N,s).
		\]

		Since $K_N=P_NK|_{Y_N}$, this is equivalent to
		\[
		(\mathrm{Id}+\mu K_N)q=P_Nf_N+\mu P_NKf_N.
		\]
		Using \(f_N=P_Nf_N+Qf_N\), we split the right-hand side as follows:
		\[
		\begin{aligned}
			(\mathrm{Id}+\mu K_N)q
			&=P_Nf_N+\mu P_NKf_N\\
			&=P_Nf_N+\mu P_NK(P_N+Q)f_N\\
			&=P_Nf_N+\mu P_NKP_Nf_N+\mu P_NKQf_N\\
			&=(\mathrm{Id}+\mu K_N)P_Nf_N+\mu P_NKQf_N.
		\end{aligned}
		\]
		Hence $(\mathrm{Id}+\mu K_N)(q-P_Nf_N)=\mu P_NKQf_N,$ which yields the decomposition
		\begin{equation}\label{eq:epd-reaction-defect}
			q=P_Nf_N+\mu(\mathrm{Id}+\mu K_N)^{-1}P_NKQf_N.
		\end{equation}
		The rank-two formula \eqref{eq:crossblock} and $\norm{f_N}_2\lesssim 1$ imply
		\[
		\norm{P_NKQf_N}_2
		\lesssim N^{-2}\bigl(|(f_N,\ell_N)|+|(f_N,\ell_{N+1})|\bigr)
		\lesssim N^{-2}.
		\]
		Since $K_N$ is positive and self-adjoint on $Y_N$,
		\[
		\norm{(\mathrm{Id}+\mu K_N)^{-1}}_{2\to2}\le1.
		\]
		Combining this with $\norm v_\infty\lesssim N\norm v_2$ for $v\in Y_N$ gives
		\begin{equation}\label{eq:epd-reaction-error}
			\norm{(r_{N,\mu}g_N)'-P_Nf_N}_\infty
			\lesssim N\mu\norm{P_NKQf_N}_2
			\lesssim \mu/N.
		\end{equation}
		
		Recall the $L^2(I)$-orthonormal basis $\{\varphi_j\}_{j=1}^{N-1}$ constructed in Section~\ref{sec:operator-representation}, and set $\mu_j=\lambda_j^{-1}$. The derivative orthogonality of this basis gives
		\[
		(\varphi_j',v')=\mu_j(\varphi_j,v),\qquad v\in X_N,\quad 1\le j\le N-1.
		\]
		Define $a_j=(\psi,\varphi_j)$. Since $\psi\in X_N$, we have
		\[
		\psi=\sum_{j=1}^{N-1}a_j\varphi_j,\qquad
		\sum_{j=1}^{N-1}a_j^2=\norm\psi_2^2=\frac25,\qquad
		\sum_{j=1}^{N-1}\mu_j a_j^2=\norm{\psi'}_2^2=1.
		\]	
		Define the coefficient functions
		\[
		\tau_k(x):=\int_I R_Nw_N(x,y)\varphi_k(y)\dd y\in X_N,\qquad 1\le k\le N-1.
		\]
		Then $R_Nw_N$ can be written as $R_Nw_N(x,y)=\sum_{k=1}^{N-1}\tau_k(x)\varphi_k(y).$
		Testing the Ritz equation \eqref{eq:variational-equation} with $v(x)\varphi_j(y)$, $v\in X_N$, gives
		\[
		\begin{aligned}
			&\sum_{k=1}^{N-1}\bigl((\tau_k',v')_I(\varphi_k,\varphi_j)_I
			+(\tau_k,v)_I(\varphi_k',\varphi_j')_I\bigr)\\
			&\qquad=(g_N',v')_I(\psi,\varphi_j)_I
			+(g_N,v)_I(\psi',\varphi_j')_I.
		\end{aligned}
		\]
		Orthogonality gives
		\[
		(\tau_j',v')_I+\mu_j(\tau_j,v)_I
		=a_j\bigl((g_N',v')_I+\mu_j(g_N,v)_I\bigr),\qquad v\in X_N.
		\]
		The definition and uniqueness of the reaction Ritz projection imply $\tau_j=a_jr_{N,\mu_j}g_N$, hence
		\[
		R_Nw_N(x,y)=\sum_{j=1}^{N-1}a_j(r_{N,\mu_j}g_N)(x)\varphi_j(y).
		\]
		To extract the term $P_Nf_N$, average the $x$ derivative against $\psi$ and define
		\[
		\varepsilon_N(x):=\sum_{j=1}^{N-1}a_j^2\bigl((r_{N,\mu_j}g_N)'(x)-P_Nf_N(x)\bigr).
		\]
		Orthogonality and \eqref{eq:epd-reaction-error} then give
		\begin{align}
			\int_I\partial_xR_Nw_N(x,y)\psi(y)\dd y
			&=\frac25P_Nf_N(x)+\varepsilon_N(x),\label{eq:endpoint-modal-average}\\
			\norm{\varepsilon_N}_\infty\lesssim \frac 1N\sum_{j=1}^{N-1}\mu_j a_j^2
			&\lesssim N^{-1}.\label{eq:endpoint-E-small}
		\end{align}
		Because $R_Nw_N$ is a polynomial, its gradient has the same supremum on $\Omega$ and $\overline\Omega$. Evaluating at $x=1$ and applying H\"older's inequality yields
		\[
		\begin{aligned}
			\norm\psi_1\norm{\nabla R_Nw_N}_\infty &\ge |\frac25P_Nf_N(1)+\varepsilon_N(1)|\\
			&\ge \frac25(P_Nf_N)(1)-\norm{\varepsilon_N}_\infty.
		\end{aligned}
		\]
		Since $\sqrt N\lesssim(P_Nf_N)(1)$ and $\norm{\varepsilon_N}_\infty\lesssim N^{-1}$, we obtain $\sqrt N\lesssim\norm{\nabla R_Nw_N}_\infty+N^{-1}$. Combining this estimate with $\norm{\nabla w_N}_\infty\lesssim1$ yields $N^{1/2}\lesssim S_{N,\infty}$. For the remaining finitely many $N\ge2$, use $R_Nv=v$ on $V_N$, hence $S_{N,\infty}\ge1$, and the bound follows for all $N\ge2$.
	\end{proof}

	\section{Upper Bound and Approximation}
	\label{sec:endpoint-upper}
	\label{subsec:endpoint-duhamel}
	We prove an $L^\infty$ bound for $\mathbb T_N$.
	By coordinate symmetry, it suffices to estimate the first-row kernels given by Lemma \ref{lem:first-row-kernels}. 
	\begin{theorem}\label{thm:epd-main}
		\begin{equation}\label{eq:audit-infty-profile}
			\|\mathbb T_N\|_{\infty\to\infty}\lesssim \sqrt N \Lambda_N^3
		\end{equation}
	\end{theorem}
	For \(a,b\in\{0,1\}\), define
	\begin{equation}\label{eq:epd-Hdef}
		h_{ab}(t):=\sup_{x\in I}\int_I|H_t^{ab}(x,y)|\dd y,
		\qquad t\ge0,
	\end{equation}
	and let \(F_N(t)=1+t/N^2\). We need to prove
	\begin{equation}\label{eq:epd-four-targets}
		\begin{aligned}
			h_{11}(t)&\lesssim \Lambda_N\sqrt N\,F_N(t)^{-1/4},&
			h_{01}(t)&\lesssim \Lambda_N(1+t)^{-1/2},\\
			h_{10}(t)&\lesssim \Lambda_N\sqrt N\,(1+t)^{-1/2},&
			h_{00}(t)&\lesssim \Lambda_N(1+t)^{-1}F_N(t)^{1/4}.
		\end{aligned}
	\end{equation}
	For $t\ge0$, set $T_N(t):=e^{-tK_N}P_N, D_1:=\mathrm{Id}, D_0:=J.$
	The kernel $H_t^{ab}$ is the integral kernel of $D_aT_N(t)D_b^*$.
	
	The first-row representation in Lemma~\ref{lem:first-row-kernels} gives
	\begin{equation}\label{eq:epd-route-time-integral}
		\|\mathbb T_N\|_{\infty\to\infty}
		\lesssim\int_0^\infty
		[h_{11}(t)h_{00}(t)+h_{10}(t)h_{01}(t)]\,\dd t.
	\end{equation}
	Let $S(t)=\mathrm e^{-tK}$ on $Y$. Extend $S(t)$ to $L^2(I)$ by setting
	\[
	S(t)f:=\mathrm e^{-tK}\left(f-\frac12\int_I f(y)\,\dd y\right),\qquad t\ge0.
	\]
	The function in parentheses belongs to $Y$, so this agrees with the original semigroup on $Y$ and removes the constant part of any other input. In particular, $S(0)$ is the mean-zero projection on $L^2(I)$.
	Let $\mathcal K_t^{ab}$ be the integral kernel of $D_aP_NS(t)D_b^*$. For $a,b\in\{0,1\}$ and $\nu\in\{0,1\}$ with $N-\nu\ge2$, define the Duhamel error kernels by
	\begin{equation}\label{eq:epd-error-kernel}
		E_{t,\nu}^{ab}(x,y):=-\rho_{N,\nu}\int_0^t
		(z^a(x))^{\top}\mathrm e^{-(t-s)M}e_{N-\nu}\,
		[D_bS(s)\ell_{N+1-\nu}](y)\,\dd s.
	\end{equation}
	Here $\rho_{N,\nu}$ is the coupling coefficient specified in \eqref{eq:epd-ranktwo}, and $e_{N-\nu}$ uses the coordinate indexing defined above.
	The kernel comparison is
	\begin{equation}\label{eq:epd-kernel-split}
		H_t^{ab}(x,y)=\mathcal K_t^{ab}(x,y)
		+\sum_{\substack{\nu\in\{0,1\}\\N-\nu\ge2}}
		E_{t,\nu}^{ab}(x,y).
	\end{equation}
	Lemma~\ref{prop:epd-mainterms} bounds the main operators $\mathcal K_t^{ab}$. To estimate $E_{t,\nu}^{ab}$, Lemma~\ref{lem:epd-exterior} controls the omitted-mode factor, and Lemma~\ref{prop:epd-Asources} controls the discrete source amplitude. Section~4.7 combines these two estimates by time convolution to bound $E_{t,\nu}^{ab}$. 
	\subsection{The Rank-Two Semigroup Comparison}	
	Using \eqref{eq:Mdiag}--\eqref{eq:Moff}, for \(f\in Y\), we have
	\begin{equation}\label{eq:epd-ranktwo}
		\begin{aligned}
			P_NKQf
			&=-\sum_{\substack{\nu\in\{0,1\}\\N-\nu\ge2}}
			\rho_{N,\nu}(f,\ell_{N+1-\nu})_{L^2(I)}\ell_{N-1-\nu},\\
			&\qquad 0<\rho_{N,\nu}\lesssim N^{-2}.
		\end{aligned}
	\end{equation}
	Writing $S(t)=\mathrm e^{-tK}$ and $T_N(t)=\mathrm e^{-tK_N}P_N$, and using $K_N=P_NKP_N|_{Y_N}$, variation of constants gives
	\begin{equation}\label{eq:epd-duhamel}
		T_N(t)
		=P_NS(t)+\int_0^t
		\mathrm e^{-(t-s)K_N}P_NKQ\,S(s)\dd s.
	\end{equation}
	For $a\in\{0,1\}$, $2\le k\le N$, and $\tau\ge0$, the function and matrix representations are related by
	\[
	[D_a\mathrm e^{-\tau K_N}\ell_{k-1}](x)
	=(z^a(x))^{\top}\mathrm e^{-\tau M}e_k.
	\]
	Applying this identity to \eqref{eq:epd-duhamel} and using \eqref{eq:epd-ranktwo} and the self-adjointness of $S(s)$ gives the kernel decomposition \eqref{eq:epd-kernel-split}, with the error kernels in \eqref{eq:epd-error-kernel}.
	
	\subsection{Jacobi Kernel and Projection Estimates}
	Let $P_m^{(\alpha,\beta)}$ denote the classical Jacobi polynomial normalized by $P_m^{(\alpha,\beta)}(1)=\binom{m+\alpha}{m}$. For $\alpha\in\{0,1,2\}$, define the orthonormal Jacobi polynomials by
	\[
	p_m^\alpha(x):=
	\frac{P_m^{(\alpha,\alpha)}(x)}
	{\left(\int_I [P_m^{(\alpha,\alpha)}(y)]^2(1-y^2)^\alpha\,\mathrm dy\right)^{1/2}},
	\qquad m\ge0.
	\]
	Then \(\{p_j^\alpha(x)\}\) is orthonormal with respect to the weight \((1-x^2)^\alpha\):
	\[
	\int_I p_m^\alpha(x)p_j^\alpha(x)(1-x^2)^\alpha\,\mathrm dx=\delta_{mj}.
	\]
	In particular, $P_m^{(0,0)}=L_m$ and $p_m^0=\ell_m$.
	%
	For $\alpha\in\{0,1\}$, define
	\begin{equation*}
		\mathcal K_n^\alpha(x,y):=\sum_{m=0}^np_m^\alpha(x)p_m^\alpha(y),
		\qquad
		\varrho_n(\theta):=\sin\theta+(n+1)^{-1}.
	\end{equation*}
	
	\begin{lemma}\label{lem:epd-Jacobi}
		For $0\le\theta,\phi\le\pi$ and $\alpha=0,1$,
		\begin{equation}\label{eq:epd-Jacobi-kernel}
			|\mathcal K_n^\alpha(\cos\theta,\cos\phi)|
			\lesssim \frac{[\varrho_n(\theta)\varrho_n(\phi)]^{-\alpha-1/2}}
			{\abs{\theta-\phi}+(n+1)^{-1}}.
		\end{equation}
		Moreover, for $m\ge1$ and $j=1,2$,
		\begin{equation}\label{eq:epd-Leg-derivatives}
			|\ell_m^{(j)}(\cos\theta)|
			\lesssim m^j(\sin\theta+m^{-1})^{-j-1/2}.
		\end{equation}
		For $k\ge2$ and $b=\sin\theta$,
		\begin{equation}\label{eq:epd-leg-diff}
			|L_k(\cos\theta)-L_{k-2}(\cos\theta)|\lesssim k^{-1/2}\sqrt b.
		\end{equation}
		In particular $\norm{\eta_k}_\infty+\norm{\eta_k}_1\lesssim 1/k$.
	\end{lemma}

	\begin{proof}
		We first prove the kernel estimate. 
		The endpoint-regularized Jacobi estimate obtained from the interior Darboux and endpoint Bessel formulas in \cite{DLMF} is
		\begin{equation*}
			|p_m^\alpha(\cos\theta)|
			\lesssim \bigl(\sin\theta+(m+1)^{-1}\bigr)^{-\alpha-1/2}.
		\end{equation*}
		For $0\le m\le n$, this gives
		\begin{subequations}
			\begin{equation}\label{eq:epd-pm-alpha-bound}
				|p_m^\alpha(\cos\theta)|
				\lesssim \varrho_n(\theta)^{-\alpha-1/2},
				\quad |p_n^\alpha(\cos\theta)|+|p_{n+1}^\alpha(\cos\theta)|\lesssim \varrho_n(\theta)^{-\alpha-1/2}
			\end{equation}
			\begin{equation}\label{eq:epd-pm-alpha-sum-bound}
				\sum_{m=0}^n|p_m^\alpha(\cos\theta)|^2
				\lesssim (n+1)\varrho_n(\theta)^{-2\alpha-1}.
			\end{equation}
		\end{subequations}
		If $|\theta-\phi|\le(n+1)^{-1}$, the Cauchy--Schwarz inequality yields
		\begin{align*}
			|\mathcal K_n^\alpha(\cos\theta,\cos\phi)|
			&\le
			\left(\sum_{m=0}^n|p_m^\alpha(\cos\theta)|^2\right)^{1/2}
			\left(\sum_{m=0}^n|p_m^\alpha(\cos\phi)|^2\right)^{1/2}\\
			&\lesssim (n+1)
			[\varrho_n(\theta)\varrho_n(\phi)]^{-\alpha-1/2}\\
			&\lesssim
			\frac{[\varrho_n(\theta)\varrho_n(\phi)]^{-\alpha-1/2}}
			{|\theta-\phi|+(n+1)^{-1}}.
		\end{align*}
		We next consider $|\theta-\phi|>(n+1)^{-1}$. Let $x=\cos\theta$, $y=\cos\phi$. The Christoffel--Darboux formula \cite{DLMF} gives
		\begin{equation}\label{eq:Christoffel--Darboux-formula}
			\mathcal K_n^\alpha(x,y)=\frac{\gamma_n^\alpha}{x-y}
			\left(
			p_{n+1}^\alpha(x)p_n^\alpha(y)
			-p_n^\alpha(x)p_{n+1}^\alpha(y)
			\right),
			\qquad |\gamma_n^\alpha|\lesssim 1.
		\end{equation}


		If $\theta$ and $\phi$ satisfy $\frac{\pi}{2}\le\theta+\phi\le\frac{3\pi}{2},$ we have
		\begin{equation}\label{eq:cos-bound}
			|\theta-\phi|\lesssim 2\sin\frac{\theta+\phi}{2}
			\sin\frac{|\theta-\phi|}{2}
			=|\cos\theta-\cos\phi|.
		\end{equation}
		Substituting $x=\cos\theta$, $y=\cos\phi$ in \eqref{eq:Christoffel--Darboux-formula} and using \eqref{eq:epd-pm-alpha-bound} and \eqref{eq:cos-bound} gives
		\begin{align*}
			|\mathcal K_n^\alpha(\cos\theta,\cos\phi)|
			&\lesssim \frac{[\varrho_n(\theta)\varrho_n(\phi)]^{-\alpha-1/2}}
			{|\theta-\phi|}\\
			&\lesssim
			\frac{[\varrho_n(\theta)\varrho_n(\phi)]^{-\alpha-1/2}}
			{|\theta-\phi|+(n+1)^{-1}}.
		\end{align*}
		
		It remains to prove \eqref{eq:epd-Jacobi-kernel} in the two off-diagonal endpoint regions
		\begin{equation*}
			|\theta-\phi|>(n+1)^{-1},\qquad
			\theta+\phi<\frac{\pi}{2}
			\quad\text{or}\quad
			\theta+\phi>\frac{3\pi}{2}.
		\end{equation*}
		Define $c_n:=\frac{p_{n+1}^\alpha(1)}{p_n^\alpha(1)}.$
		The Jacobi connection formulas in \cite{DLMF} give
		\begin{align*}
			p_{n+1}^\alpha(x)-c_np_n^\alpha(x)
			&=\gamma_{n,+}^\alpha(1-x)\sqrt{n+1}\,
			P_n^{(\alpha+1,\alpha)}(x),\\
			p_{n+1}^\alpha(x)+c_np_n^\alpha(x)
			&=\gamma_{n,-}^\alpha(1+x)\sqrt{n+1}\,
			P_n^{(\alpha,\alpha+1)}(x),
		\end{align*}
		where $\gamma_{n,+}$ and $\gamma_{n,-}$ satisfy $|\gamma_{n,+}^\alpha| + |\gamma_{n,-}^\alpha|\lesssim 1.$
		Applying the endpoint Jacobi estimates to these two identities gives
		\begin{align*}
			|p_{n+1}^\alpha(\cos\theta)-c_np_n^\alpha(\cos\theta)|
			&\lesssim
			\bigl(\theta+(n+1)^{-1}\bigr)
			\varrho_n(\theta)^{-\alpha-1/2},
			\,0\le\theta\le\frac{\pi}{2},\\
			|p_{n+1}^\alpha(\cos\theta)+c_np_n^\alpha(\cos\theta)|
			&\lesssim
			\bigl(\pi-\theta+(n+1)^{-1}\bigr)
			\varrho_n(\theta)^{-\alpha-1/2},
			\,\frac{\pi}{2}\le\theta\le\pi.
		\end{align*}		
		If $\theta+\phi<\pi/2$, decompose the Christoffel--Darboux numerator as
		\begin{align*}
			&p_{n+1}^\alpha(x)p_n^\alpha(y)
			-p_n^\alpha(x)p_{n+1}^\alpha(y)\\
			&=
			\bigl(p_{n+1}^\alpha(x)-c_np_n^\alpha(x)\bigr)p_n^\alpha(y)
			-p_n^\alpha(x)
			\bigl(p_{n+1}^\alpha(y)-c_np_n^\alpha(y)\bigr).
		\end{align*}
		The preceding endpoint estimates imply
		\begin{align*}
			&\left|
			p_{n+1}^\alpha(x)p_n^\alpha(y)
			-p_n^\alpha(x)p_{n+1}^\alpha(y)
			\right|\\
			&\lesssim
			\bigl(\theta+\phi+2(n+1)^{-1}\bigr)
			[\varrho_n(\theta)\varrho_n(\phi)]^{-\alpha-1/2}\\
			&\lesssim
			(\theta+\phi)
			[\varrho_n(\theta)\varrho_n(\phi)]^{-\alpha-1/2}.
		\end{align*}
		In this region, $(\theta+\phi)|\theta-\phi|\lesssim|\cos\theta-\cos\phi|.$
		Consequently,
		\begin{equation*}
			|\mathcal K_n^\alpha(\cos\theta,\cos\phi)|
			\lesssim
			\frac{[\varrho_n(\theta)\varrho_n(\phi)]^{-\alpha-1/2}}
			{|\theta-\phi|+(n+1)^{-1}}.
		\end{equation*}
		
		The case $\theta+\phi>3\pi/2$ follows by replacing $(\theta,\phi)$ with $(\pi-\theta,\pi-\phi)$, using $\mathcal K_n^\alpha(-x,-y)=\mathcal K_n^\alpha(x,y)$.
		This proves \eqref{eq:epd-Jacobi-kernel}.
		
		We next prove \eqref{eq:epd-Leg-derivatives}. 
		The Jacobi derivative identity \cite{DLMF} gives
		\begin{equation*}
			\frac{d^j}{dx^j}P_m^{(0,0)}(x)=2^{-j}\frac{(m+j)!}{m!}P_{m-j}^{(j,j)}(x),
			\qquad j=1,2, \quad m\ge j.
		\end{equation*}
		After normalization, this identity becomes
		\begin{equation*}
			\ell_m^{(j)}(x)=a_{m,j}p_{m-j}^j(x),
			\qquad |a_{m,j}|\lesssim m^j.
		\end{equation*}
		The endpoint-regularized Jacobi estimate with parameter $j$ yields
		\begin{align*}
			|\ell_m^{(j)}(\cos\theta)|
			&\lesssim
			m^j\bigl(\sin\theta+(m-j+1)^{-1}\bigr)^{-j-1/2}\\
			&\lesssim
			m^j(\sin\theta+m^{-1})^{-j-1/2}.
		\end{align*}
		If $m<j$, then $\ell_m^{(j)}=0$, so the estimate also holds. This proves \eqref{eq:epd-Leg-derivatives}.
		
		
		Finally, let $x=\cos\theta$ and $b=\sin\theta$. The Legendre identity
		\begin{equation*}
			L_k(x)-L_{k-2}(x)
			=
			-\frac{2k-1}{k(k-1)}(1-x^2)L_{k-1}'(x)
		\end{equation*}
		and the case $j=1$ of \eqref{eq:epd-Leg-derivatives} give
		\begin{equation*}
			|L_{k-1}'(\cos\theta)|
			=
			\sqrt{\frac{2}{2k-1}}\,
			|\ell_{k-1}'(\cos\theta)|
			\lesssim
			k^{1/2}(b+k^{-1})^{-3/2}.
		\end{equation*}
		Using $(2k-1)/(k(k-1))\lesssim 1/k$, we obtain
		\begin{align*}
			|L_k(\cos\theta)-L_{k-2}(\cos\theta)|
			&\lesssim
			k^{-1/2}b^2(b+k^{-1})^{-3/2}\\
			&\lesssim
			k^{-1/2}\sqrt b.
		\end{align*}
		The last step uses $b^2\le \sqrt b\,(b+k^{-1})^{3/2}$. The definition of $\eta_k$ then gives
		\begin{equation*}
			|\eta_k(\cos\theta)|
			\lesssim \frac{1}{k}\sqrt{\sin\theta}
			\lesssim \frac{1}{k},
		\end{equation*}
		and hence we obtain $\|\eta_k\|_\infty+\|\eta_k\|_1 \lesssim \frac{1}{k}$. This completes the proof.
	\end{proof}

	For the fixed degree $N$, let $\Pi_N$ denote the one-dimensional energy projection onto $X_N$, defined by
	\[
	\Pi_N f=-\sum_{k=2}^N\eta_k\int_I\eta_k''f.
	\]
	For $f\in H^1_0(I)$, integration by parts and \eqref{eq:eta-der} give
	\[
	\Pi_N f=JP_Nf',\qquad (\Pi_Nf)'=P_Nf'.
	\]
	On $L^2(I,(1-x^2)^{-1}\,\mathrm dx)$, $\Pi_N$ is also the orthogonal projection onto $X_N$, with the kernel representation
	\begin{equation}\label{eq:epd-E-kernel}
		\Pi_N f(x)=(1-x^2)\int_I K_{N-2}^{1}(x,y)f(y)\,\mathrm dy.
	\end{equation}
	The same formula defines its action on $L^\infty(I)$. We also have $JP_N=\Pi_NJ$ on mean-zero inputs.
	
	\begin{lemma}\label{lem:epd-Lebesgue}
		Let $w(x)=(1-x^2)^{-1/4}$. Then
		\begin{equation}\label{eq:epd-weighted-Lebesgue}
			\sup_x\int_I |k_N(x,y)|w(y)\dd y\lesssim \sqrt N \Lambda_N,
			\qquad \norm{\Pi_N}_{\infty\to\infty}\lesssim \Lambda_N.
		\end{equation}
		In particular $\norm{P_N}_{\infty\to\infty}\lesssim \sqrt N \Lambda_N$.
	\end{lemma}
	\begin{proof}
		We first prove the weighted kernel estimate for $k_N$ and then the $L^\infty$-bound for $\Pi_N$. Recall that
		\begin{equation*}
			k_N(x,y)=\mathcal K_{N-1}^0(x,y)-\ell_0(x)\ell_0(y),
			\qquad \ell_0=2^{-1/2}.
		\end{equation*}
		Let $x=\cos\theta$ and $y=\cos\phi$. The change of variables $y=\cos\phi$ gives
		\begin{equation*}
			w(\cos\phi)|\dd y|
			=(\sin\phi)^{-1/2}\sin\phi\dd\phi
			=\sqrt{\sin\phi}\dd\phi.
		\end{equation*}
		Applying \eqref{eq:epd-Jacobi-kernel} with $\alpha=0$ and $n=N-1$, we obtain
		\begin{align*}
			\int_I|\mathcal K_{N-1}^0(x,y)|w(y)\dd y
			&\lesssim \frac{1}{\sqrt{\sin\theta+N^{-1}}}
			\int_0^\pi
			\frac{\sqrt{\sin\phi}}
			{\sqrt{\sin\phi+N^{-1}}}
			\frac{\dd\phi}{|\theta-\phi|+N^{-1}}\\
			&\lesssim \sqrt N\int_0^\pi
			\frac{\dd\phi}{|\theta-\phi|+N^{-1}}\\
			&\lesssim\sqrt N\left\{
			\log(1+N\theta)+\log\bigl(1+N(\pi-\theta)\bigr)
			\right\}\\
			&\lesssim \sqrt N\,\Lambda_N.
		\end{align*}
		The constant Legendre mode satisfies
		\begin{equation*}
			\sup_{x\in I}\int_I|\ell_0(x)\ell_0(y)|w(y)\dd y
			=\frac12\int_I(1-y^2)^{-1/4}\dd y
			\lesssim 1.
		\end{equation*}
		Since $w\ge1$ and $\Lambda_N\ge1$, the preceding estimates give
		\[
		\|P_N\|_{\infty\to\infty}
		\le\sup_{x\in I}\int_I|k_N(x,y)|w(y)\,\dd y
		\lesssim\sqrt N\,\Lambda_N.
		\]
		
		We next estimate $\Pi_N$. For $0\le\theta\le\pi$, define
		\begin{equation*}
			a_\theta:=\sin\theta+N^{-1}.
		\end{equation*}
		For $f\in L^\infty(I)$, \eqref{eq:epd-E-kernel} and \eqref{eq:epd-Jacobi-kernel} with $\alpha=1$ give
		\begin{align*}
			|\Pi_N f(\cos\theta)|
			&\lesssim
			\|f\|_\infty
			\frac{\sin^2\theta}{a_\theta^{3/2}}
			\int_0^\pi
			\frac{\sin\phi\dd\phi}
			{(\sin\phi+N^{-1})^{3/2}(|\theta-\phi|+N^{-1})}\\
			&\lesssim \|f\|_\infty I(\theta),
		\end{align*}
		where $I(\theta):=\sqrt{a_\theta}\int_0^\pi\frac{\sin\phi\dd\phi} {(\sin\phi+N^{-1})^{3/2} (|\theta-\phi|+N^{-1})}.$
		Since $I(\pi-\theta)=I(\theta)$, it is enough to consider $0\le\theta\le\pi/2$. 
		Define
		\begin{align*}
			E_1(\theta)&:=\left\{\phi\in[0,\pi]:
			|\phi-\theta|\le\frac{a_\theta}{2}\right\},\\
			E_2(\theta)&:=[0,\pi/2]\setminus E_1(\theta),\qquad
			E_3(\theta):=(\pi/2,\pi]\setminus E_1(\theta),\\
			I_j(\theta)&:=\sqrt{a_\theta}\int_{E_j(\theta)}\frac{\sin\phi\dd\phi} 	{(\sin\phi+N^{-1})^{3/2}(|\theta-\phi|+N^{-1})},
			\qquad j=1,2,3.
		\end{align*}
		Then we have $I(\theta)=I_1(\theta)+I_2(\theta)+I_3(\theta).$
		
		For $\phi\in E_1(\theta)$, the Lipschitz bound $|\sin\phi-\sin\theta|\le|\phi-\theta|$ yields
		\begin{equation*}
			\frac{a_\theta}{2}
			\le\sin\phi+N^{-1}
			\le\frac{3a_\theta}{2}.
		\end{equation*}
		Consequently,
		\begin{align*}
			I_1(\theta)
			&\lesssim \int_{E_1(\theta)}
			\frac{\dd\phi}{|\theta-\phi|+N^{-1}}\\
			&\lesssim \int_0^{a_\theta/2}\frac{\dd u}{u+N^{-1}}
			\lesssim\log\left(1+\frac{Na_\theta}{2}\right)
			\lesssim \Lambda_N.
		\end{align*}	
		On $[0,\pi/2]$, we have $a_\theta\eqsim\theta+N^{-1}, \sin\phi+N^{-1} \eqsim\phi+N^{-1}.$
		Splitting $E_2(\theta)$ at $\phi=\theta$, using $|\phi-\theta|>a_\theta/2$, and setting $u=\phi-\theta$ in the second integral, we obtain
		\begin{align*}
			I_2(\theta)
			&\lesssim \frac{\sqrt{a_\theta}}{a_\theta}
			\int_0^\theta(\phi+N^{-1})^{-1/2}\dd\phi
			+
			\sqrt{a_\theta}
			\int_{a_\theta/2}^{\infty}u^{-3/2}\dd u\\
			&\lesssim \frac{\sqrt{a_\theta}\sqrt{\theta+N^{-1}}}{a_\theta}
			+1
			\lesssim 1.
		\end{align*}
		
		It remains to estimate $I_3(\theta)$. If $0\le\theta\le\pi/4$, then we have
		\begin{equation*}
			|\phi-\theta|\ge\frac{\pi}{4},
			\qquad \phi\in(\pi/2,\pi].
		\end{equation*}
		If $\pi/4<\theta\le\pi/2$, then $a_\theta\ge\sin(\pi/4)$, and the definition of $E_3(\theta)$ gives
		\begin{equation*}
			|\phi-\theta|>\frac{a_\theta}{2}\ge\frac{\sin(\pi/4)}{2}.
		\end{equation*}
		Hence $|\phi-\theta|+N^{-1}\ge c$ on $E_3(\theta)$ in both cases. Using $\sin\phi\le\sin\phi+N^{-1}$ and then setting $u=\pi-\phi$, we obtain
		\begin{align*}
			I_3(\theta)
			&\lesssim
			\sqrt{a_\theta}
			\int_{\pi/2}^{\pi}
			(\sin\phi+N^{-1})^{-1/2}\dd\phi\\
			&\lesssim
			\int_0^{\pi/2}(u+N^{-1})^{-1/2}\dd u
			\lesssim 1.
		\end{align*}
		The bounds for $I_1$, $I_2$, $I_3$ give $I(\theta)\lesssim\Lambda_N$,  proving the estimate for $\Pi_N$. 
	\end{proof}

	\subsection{Fractional Estimates and Infinite-Dimensional Semigroups}
	
	Let $\mathbb C_4=\mathbb R/(4\mathbb Z)$ denote the circle of length $4$,
	equipped with Lebesgue measure over one period. For a function $f$ defined on $I$, let $f_{\rm e}$ denote its even reflection at both endpoints, regarded as
	a function on $\mathbb C_4$. Define the distance on this circle by $d_4(x,y):=\min_{r\in\mathbb Z}|x-y+4r|.$ 
	Let $B^{1/2}_{1,1}=B^{1/2}_{1,1}(\mathbb C_4)$ denote the usual
	inhomogeneous Besov space. Define its norm here by the equivalent
	fractional Sobolev expression
	\[
	\norm f_{B^{1/2}_{1,1}}:=\norm f_{L^1(\mathbb C_4)}
	+\int_{\mathbb C_4}\int_{\mathbb C_4}
	\frac{|f(x)-f(y)|}{d_4(x,y)^{3/2}}\dd x\dd y.
	\]
	
	\begin{lemma}\label{lem:epd-fractional}
		For every polynomial $p$ of degree at most $N$,
		\begin{equation}\label{eq:epd-fractional}
			\norm{p_{\rm e}}_{B^{1/2}_{1,1}}
			\lesssim \sqrt N\int_I |p(x)|(1-x^2)^{-1/4}\dd x.
		\end{equation}
	\end{lemma}
	
	\begin{proof}
		We first prove the regularized weighted Bernstein estimate
		\begin{equation}\label{eq:epd-reg-Bernstein}
			\int_I |p'(x)|(1-x^2+N^{-2})^{1/4}\dd x
			\lesssim N\int_I |p(x)|(1-x^2)^{-1/4}\dd x.
		\end{equation}
		This estimate will be used below to control the short-distance part of the
		Besov seminorm.	Let
		$
		\mathbb C_{2\pi}:=\mathbb R/(2\pi\mathbb Z),
		d_{2\pi}(\theta,\phi):=
		\min_{r\in\mathbb Z}|\theta-\phi+2\pi r|.
		$
		On \(\mathbb C_{2\pi}\), use the even \(2\pi\)-periodic extension
		\(\varrho_{N-1}(\theta)=|\sin\theta|+N^{-1}\) of the weight defined above.
		Define
		$
		q(\theta):=p(\cos\theta).
		$
		Then \(q\) is a trigonometric polynomial of degree at most \(N\). For
		\(f,g\in L^1(\mathbb C_{2\pi})\), define
		\[
		(f*g)(\theta):=\frac1{2\pi}\int_{\mathbb C_{2\pi}}
		f(\theta-\phi)g(\phi)\dd\phi.
		\]
		Choose a fixed even function
		\(\chi_{\rm ang}\in C_c^\infty((-2,2))\) satisfying
		\(\chi_{\rm ang}=1\) on \([-1,1]\), and define
		\[
		\Phi_N(\theta):=
		\sum_{k\in\mathbb Z}\chi_{\rm ang}(k/N)\mathrm e^{ik\theta}.
		\]
		The Fourier support of \(\Phi_N\) is contained in
		\(\{k\in\mathbb Z:|k|<2N\}\), while
		\(\chi_{\rm ang}(k/N)=1\) for \(|k|\le N\). Hence
		\[
		q=\Phi_N*q,\qquad q'=\Phi_N'*q.
		\]
		For every fixed integer \(A\ge4\), repeated summation by parts gives
		\[
		|\Phi_N(\theta)|
		\le C_AN\bigl(1+Nd_{2\pi}(\theta,0)\bigr)^{-A},
		\]
		\[
		|\Phi_N'(\theta)|
		\le C_AN^2\bigl(1+Nd_{2\pi}(\theta,0)\bigr)^{-A}.
		\]
		The Lipschitz estimate
		\(\bigl||\sin\theta|-|\sin\phi|\bigr|\le d_{2\pi}(\theta,\phi)\)
		implies
		\[
		\frac{\varrho_{N-1}(\theta)^{1/2}}{\varrho_{N-1}(\phi)^{1/2}}
		\lesssim \bigl(1+Nd_{2\pi}(\theta,\phi)\bigr)^{1/2}.
		\]
		Using \(q'=\Phi_N'*q\), and Fubini's theorem, we obtain
		\begin{align*}
			\int_{\mathbb C_{2\pi}}|q'(\theta)|\varrho_{N-1}(\theta)^{1/2}\dd\theta
			&\le C_AN^2\int_{\mathbb C_{2\pi}}|q(\phi)|\varrho_{N-1}(\phi)^{1/2}\\
			&\quad\times
			\int_{\mathbb C_{2\pi}}
			\bigl(1+Nd_{2\pi}(\theta,\phi)\bigr)^{-A+1/2}
			\dd\theta\dd\phi.
		\end{align*}
		The change of variables $u=N d_{2\pi}(\theta,\phi)$ yields
		\[
		\sup_{\phi\in\mathbb C_{2\pi}}
		\int_{\mathbb C_{2\pi}}
		\bigl(1+Nd_{2\pi}(\theta,\phi)\bigr)^{-A+1/2}\dd\theta
		\le \frac{C_A}{N}.
		\]
		Consequently,
		\begin{equation}\label{eq:epd-angular-Bernstein}
			\int_{\mathbb C_{2\pi}}|q'(\theta)|\varrho_{N-1}(\theta)^{1/2}\dd\theta
			\lesssim N\int_{\mathbb C_{2\pi}}|q(\theta)|\varrho_{N-1}(\theta)^{1/2}\dd\theta.
		\end{equation}
		
		It remains to remove \(N^{-1}\) from the weight on the right-hand side.
		For \(0<\delta<1\), define
		\[
		A_\delta:=
		\left\{\theta\in\mathbb C_{2\pi}:
		\min\{d_{2\pi}(\theta,0),d_{2\pi}(\theta,\pi)\}
		<\frac{\delta}{N}\right\}.
		\]
		The identity \(q=\Phi_N*q\), the kernel estimate for \(\Phi_N\), and the
		weight-ratio estimate give
		\[
		\sup_{\phi\in\mathbb C_{2\pi}}
		\int_{A_\delta}
		|\Phi_N(\theta-\phi)|
		\frac{\varrho_{N-1}(\theta)^{1/2}}{\varrho_{N-1}(\phi)^{1/2}}\dd\theta
		\lesssim \delta.
		\]
		Integrating the reproducing identity over \(A_\delta\) therefore yields
		\[
		\int_{A_\delta}|q(\theta)|\varrho_{N-1}(\theta)^{1/2}\dd\theta
		\lesssim \delta
		\int_{\mathbb C_{2\pi}}|q(\theta)|\varrho_{N-1}(\theta)^{1/2}\dd\theta.
		\]
		Fix \(\delta>0\) sufficiently small that the \(A_\delta\) contribution is at most one half of the full weighted norm. On
		\(\mathbb C_{2\pi}\setminus A_\delta\), one has
		\[
		|\sin\theta|\ge c_\delta N^{-1},
		\qquad
		\varrho_{N-1}(\theta)^{1/2}\lesssim |\sin\theta|^{1/2}.
		\]
		It follows that
		\begin{align*}
			\int_{\mathbb C_{2\pi}}|q(\theta)|\varrho_{N-1}(\theta)^{1/2}\dd\theta
			&=
			\int_{A_\delta}|q(\theta)|\varrho_{N-1}(\theta)^{1/2}\dd\theta\\
			&\quad+\int_{\mathbb C_{2\pi}\setminus A_\delta}|q(\theta)|\varrho_{N-1}(\theta)^{1/2}\dd\theta\\
			&\le \frac12\int_{\mathbb C_{2\pi}}|q(\theta)|\varrho_{N-1}(\theta)^{1/2}\dd\theta\\
			&\quad+C_\delta\int_{\mathbb C_{2\pi}}
			|q(\theta)||\sin\theta|^{1/2}\dd\theta.
		\end{align*}
		Absorbing the first term on the right gives
		\begin{equation}\label{eq:epd-remove-regularization}
			\int_{\mathbb C_{2\pi}}|q(\theta)|\varrho_{N-1}(\theta)^{1/2}\dd\theta
			\lesssim \int_{\mathbb C_{2\pi}}
			|q(\theta)||\sin\theta|^{1/2}\dd\theta.
		\end{equation}
		Using \(q'(\theta)=-p'(\cos\theta)\sin\theta\) and \(x=\cos\theta\), we have
		\[
		\int_{\mathbb C_{2\pi}}|q'(\theta)|\varrho_{N-1}(\theta)^{1/2}\dd\theta
		=
		2\int_I|p'(x)|
		\bigl(\sqrt{1-x^2}+N^{-1}\bigr)^{1/2}\dd x,
		\]
		\[
		\int_{\mathbb C_{2\pi}}|q(\theta)||\sin\theta|^{1/2}\dd\theta
		=
		2\int_I|p(x)|(1-x^2)^{-1/4}\dd x.
		\]
		Moreover,
		\[
		(1-x^2+N^{-2})^{1/4}
		\le \bigl(\sqrt{1-x^2}+N^{-1}\bigr)^{1/2}
		\le 2^{1/4}(1-x^2+N^{-2})^{1/4}.
		\]
		Substituting \eqref{eq:epd-remove-regularization} into
		\eqref{eq:epd-angular-Bernstein} and using these identities proves
		\eqref{eq:epd-reg-Bernstein}.
		
		We now estimate the fractional seminorm in $B^{1/2}_{1,1}(\mathbb C_4)$. Define
		\[
		a(x):=\frac{c}{N}\sqrt{1-x^2+N^{-2}},\qquad x\in I,
		\]
		and extend \(a\) to \(\mathbb C_4=\mathbb R/(4\mathbb Z)\) by the same even reflections defining \(p_{\rm e}\). Since
		\[
		|a'(x)|
		=\frac{c|x|}{N\sqrt{1-x^2+N^{-2}}}
		\le c,
		\]
		we may fix \(c>0\) sufficiently small so that
		\[
		|a(x)-a(y)|\le\frac14d_4(x,y),
		\qquad x,y\in\mathbb C_4.
		\]
		Define the long-distance and short-distance contributions by
		\[
		\mathcal I_{\rm far}:=
		\int_{\mathbb C_4}
		\int_{\{y:\,d_4(x,y)\ge a(x)\}}
		\frac{|p_{\rm e}(x)-p_{\rm e}(y)|}
		{d_4(x,y)^{3/2}}\dd y\dd x,
		\]
		\[
		\mathcal I_{\rm near}:=
		\int_{\mathbb C_4}
		\int_{\{y:\,d_4(x,y)<a(x)\}}
		\frac{|p_{\rm e}(x)-p_{\rm e}(y)|}
		{d_4(x,y)^{3/2}}\dd y\dd x.
		\]
		
		To estimate $\mathcal I_{\rm far}$, we use $|p_{\rm e}(x)-p_{\rm e}(y)|\le |p_{\rm e}(x)|+|p_{\rm e}(y)|$ together with the estimate
		\[
		\int_{\{y:\,d_4(x,y)\ge a(x)\}}
		d_4(x,y)^{-3/2}\dd y
		\lesssim a(x)^{-1/2}.
		\]
		If \(d_4(x,y)\ge a(x)\), the Lipschitz estimate for \(a\) gives
		\[
		a(y)\le a(x)+\frac14d_4(x,y)
		\le\frac54d_4(x,y),
		\]
		and hence \(d_4(x,y)\ge\frac45a(y)\). After interchanging \(x\) and \(y\),
		the second term satisfies the same bound. Therefore
		\begin{align*}
			\mathcal I_{\rm far}
			&\lesssim \int_{\mathbb C_4}|p_{\rm e}(x)|a(x)^{-1/2}\dd x\\
			&\lesssim \sqrt N\int_I
			|p(x)|(1-x^2+N^{-2})^{-1/4}\dd x\\
			&\lesssim \sqrt N\int_I
			|p(x)|(1-x^2)^{-1/4}\dd x.
		\end{align*}
		
		For \(\mathcal I_{\rm near}\), let \(\gamma_{x,y}\) be the shorter arc in
		\(\mathbb C_4\) joining \(x\) and \(y\). The fundamental theorem of
		calculus gives
		\[
		|p_{\rm e}(x)-p_{\rm e}(y)|
		\le\int_{\gamma_{x,y}}|p_{\rm e}'(z)|\dd z.
		\]
		If \(z\in\gamma_{x,y}\) and \(d_4(x,y)<a(x)\), then
		\[
		|a(z)-a(x)|
		\le\frac14d_4(x,z)
		<\frac14a(x),
		\]
		so that $\frac34a(x)\le a(z)\le\frac54a(x).$
		After inserting the fundamental theorem of calculus and interchanging
		the \(x,y,z\) integrations, the pairs \((x,y)\) whose shorter arc contains
		\(z\) can be parametrized by two nonnegative distances \(u\) and \(v\)
		satisfying \(u+v\le\frac43a(z)\). Hence
		\begin{align*}
			\mathcal I_{\rm near}
			&\lesssim \int_{\mathbb C_4}|p_{\rm e}'(z)|
			\int_{\substack{u,v\ge0\\u+v\le\frac43a(z)}}
			(u+v)^{-3/2}\dd u\dd v\dd z\\
			&\lesssim\int_{\mathbb C_4}|p_{\rm e}'(z)|
			\int_0^{\frac43a(z)}s^{-3/2}
			\left(\int_0^s\dd u\right)\dd s\dd z\\
			&\lesssim \int_{\mathbb C_4}|p_{\rm e}'(z)|a(z)^{1/2}\dd z\\
			&\lesssim \frac{1}{\sqrt N}
			\int_I|p'(x)|(1-x^2+N^{-2})^{1/4}\dd x.
		\end{align*}
		Applying \eqref{eq:epd-reg-Bernstein} to the last integral gives
		\[
		\mathcal I_{\rm near}
		\lesssim \sqrt N\int_I|p(x)|(1-x^2)^{-1/4}\dd x.
		\]
		
		Finally, the even reflection satisfies
		\[
		\|p_{\rm e}\|_{L^1(\mathbb C_4)}
		=2\int_I|p(x)|\dd x
		\le 2\int_I|p(x)|(1-x^2)^{-1/4}\dd x.
		\]
		Combining these three bounds proves \eqref{eq:epd-fractional}.
	\end{proof}

	Define the semigroup generated by the negative inverse of the Dirichlet Laplacian on \(L^2(I)\) by
	\[
	S_D(t):=\mathrm e^{-tJJ^*},\qquad t\ge0.
	\]
	Under the period-\(4\) even Neumann and odd Dirichlet extensions to
	\(\mathbb C_4\), both \(S(t)\) and \(S_D(t)\) have the Fourier multiplier
	\[
	\mathrm e^{-t/\xi^2},
	\qquad
	\xi=\frac{j\pi}{2},
	\qquad
	j\in\mathbb Z\setminus\{0\}.
	\]
	
	The frequency \(j=0\) is absent for \(S(t)\) and \(S_D(t)\). 
	Expanding the operator exponentials into power series gives
	\[
	JS(t)=S_D(t)J\quad\text{on }Y,
	\qquad
	S(t)J^*=J^*S_D(t)\quad\text{on }L^2(I).
	\]
	For a finite signed or complex Borel measure \(\mu\) on \(\mathbb C_4\), let \(|\mu|\) denote its total variation measure and define the norm $\norm{\mu}_{\mathrm{meas}}:=|\mu|(\mathbb C_4).$
	\begin{lemma}\label{lem:epd-Fourier}
		For $t\ge0$ and $p=1,\infty$,
		\begin{align}
			\norm{S(t)}_{p\to p}+\norm{S_D(t)}_{p\to p}&\lesssim 1,\label{eq:epd-S0}\\
			\norm{JS(t)}_{p\to p}+\norm{S(t)J^*}_{p\to p}
			&\lesssim (1+t)^{-1/2},\label{eq:epd-S1}\\
			\norm{JS(t)J^*}_{p\to p}&\lesssim (1+t)^{-1}.\label{eq:epd-S2}
		\end{align}
		For $t\ge1$ and $f$ with $f_{\rm e}\in B^{1/2}_{1,1}$,
		\begin{equation}\label{eq:epd-Besov-highpass}
			\norm{S(t)f}_1\lesssim t^{-1/4}\norm{f_{\rm e}}_{B^{1/2}_{1,1}}.
		\end{equation}
		If the even reflection has a distributional second derivative which is
		a finite measure, then
		\begin{equation}\label{eq:epd-second-highpass}
			\norm{S(t)f}_1\lesssim t^{-1}\norm{(f_{\rm e})''}_{\mathrm{meas}}.
		\end{equation}
	\end{lemma}
	\begin{proof}
		We first represent the operators as Fourier multipliers on
		\(\mathbb C_4=\mathbb R/(4\mathbb Z)\). For \(g\in L^1(\mathbb C_4)\), set
		\[
		\xi_k:=\frac{k\pi}{2},\qquad
		\widehat g(k):=\frac14\int_0^4g(x)\mathrm e^{-i\xi_kx}\,\dd x,
		\qquad k\in\mathbb Z.
		\]
		For a multiplier \(m\) defined on \(\{\xi_k:k\in\mathbb Z\}\), let \(\mathcal M_m\) denote the periodic Fourier multiplier
		\[
		\mathcal M_mg(x):=\sum_{k\in\mathbb Z}
		m(\xi_k)\widehat g(k)\mathrm e^{i\xi_kx}.
		\]
		If the periodic convolution kernel \(\mathcal K_m\) of \(\mathcal M_m\) belongs to
		\(L^1(\mathbb C_4)\), Young's inequality gives
		\[
		\|\mathcal M_m\|_{L^p(\mathbb C_4)\to L^p(\mathbb C_4)}
		\lesssim \|\mathcal K_m\|_{L^1(\mathbb C_4)},
		\qquad p\in\{1,\infty\}.
		\]
		Let \(R=\sqrt t\). Under the \(4\)-periodic even extension,
		the multiplier of \(S(t)\) is $0$ for $\xi=0$ and $\mathrm e^{-R^2/\xi^2}$ for $ \xi\ne0$.
		Under the odd Dirichlet extension to $\mathbb C_4$, $S_D(t)$ has the same Fourier multiplier on nonzero frequencies. The multipliers of \(JS(t)\),
		\(S(t)J^*\) and \(JS(t)J^*\) on nonzero frequencies are
		\[
		(i\xi)^{-1}\mathrm e^{-R^2/\xi^2},\qquad
		(-i\xi)^{-1}\mathrm e^{-R^2/\xi^2},\qquad
		\xi^{-2}\mathrm e^{-R^2/\xi^2}.
		\]
		For $\eta \ne 0$, define the three real-line symbols by
		\[
		\sigma_0(\eta):=1-\mathrm e^{-1/\eta^2},\qquad
		\sigma_1(\eta):=(i\eta)^{-1}\mathrm e^{-1/\eta^2},\qquad
		\sigma_2(\eta):=\eta^{-2}\mathrm e^{-1/\eta^2},
		\]
		and \(\sigma_0(0):=1\), \(\sigma_1(0)=\sigma_2(0):=0\). Each
		\(\sigma_\nu\) belongs to \(H^1(\mathbb R)\). 
		For every \(\sigma\in H^1(\mathbb R)\), Cauchy--Schwarz and
		Plancherel give
		\[
		\begin{aligned}
			\|\mathcal F^{-1}\sigma\|_{L^1(\mathbb R)}
			\le
			\left(\int_{\mathbb R}\frac{\dd x}{1+x^2}\right)^{1/2}
			\left\|(1+x^2)^{1/2}\mathcal F^{-1}\sigma\right\|_2  
			\lesssim \bigl(\|\sigma\|_2+\|\sigma'\|_2\bigr).
		\end{aligned}
		\]
		Hence \(\mathcal F^{-1}\sigma_\nu\in L^1(\mathbb R)\) for
		\(\nu=0,1,2\).
		
		The change of variables $\eta=\xi/R$ gives kernel bounds $C$, $CR^{-1}$, and $CR^{-2}$. These bounds are preserved under periodization.
		The identity $m_{0,R}(\xi)=1-\sigma_0(\xi/R), \xi\in\{\xi_k:k\in\mathbb Z\},$  includes \(\xi=0\), because \(m_{0,R}(0)=0\) and \(\sigma_0(0)=1\).
		Consequently, \(S(t)\) on the even extension and \(S_D(t)\) on the odd
		extension are the identity minus a periodic convolution operator whose
		kernel has uniformly bounded \(L^1\) norm. Young's inequality yields
		\[
		\|S(t)\|_{p\to p}+\|S_D(t)\|_{p\to p}\lesssim 1,
		\qquad p\in\{1,\infty\},
		\]
		which proves \eqref{eq:epd-S0} for \(t>0\). At \(t=0\),
		\[
		S(0)f=f-\frac12\int_I f(y)\,\dd y,\qquad S_D(0)f=f.
		\]
		The estimate follows since \(\|S(0)f\|_p\le2\|f\|_p\) for \(p=1,\infty\).
		
		For \(t\ge1\), the other two periodic kernel bounds give
		\[
		\begin{aligned}
			\|JS(t)\|_{p\to p}+\|S(t)J^*\|_{p\to p}
			&\lesssim R^{-1}=t^{-1/2},\\
			\|JS(t)J^*\|_{p\to p}
			&\lesssim R^{-2}=t^{-1}.
		\end{aligned}
		\]
		For \(0\le t\le1\), the boundedness of \(J\) and \(J^*\), together with
		\eqref{eq:epd-S0}, gives
		\[
		\|JS(t)\|_{p\to p}+\|S(t)J^*\|_{p\to p}\lesssim 1,
		\qquad
		\|JS(t)J^*\|_{p\to p}\lesssim 1.
		\]
		Combining the bounds for $0\le t\le1$ and $t\ge1$ proves \eqref{eq:epd-S1}--\eqref{eq:epd-S2}.
		
		We next prove \eqref{eq:epd-second-highpass}. Let
		\(\mu:=(f_{\rm e})''\), interpreted as a finite measure on
		\(\mathbb C_4\). For \(k\ne0\),
		\[
		\widehat\mu(k)=-\xi_k^2\widehat{f_{\rm e}}(k),
		\]
		and the zero Fourier mode is removed by \(S(t)\). Hence
		\[
		\widehat{(S(t)f)_{\rm e}}(k)
		=-\xi_k^{-2}\mathrm e^{-R^2/\xi_k^2}\widehat\mu(k),
		\qquad k\ne0.
		\]
		The periodic kernel associated with
		\(\xi^{-2}\mathrm e^{-R^2/\xi^2}\) has \(L^1\) norm at most
		\(CR^{-2}\). Convolution of this kernel with \(\mu\) therefore gives
		\[
		\|S(t)f\|_{L^1(I)}
		\lesssim R^{-2}\|\mu\|_{\mathrm{meas}}
		=t^{-1}\|(f_{\rm e})''\|_{\mathrm{meas}},
		\]
		which proves \eqref{eq:epd-second-highpass}.
		
		It remains to prove \eqref{eq:epd-Besov-highpass}. Assume \(t\ge1\), so
		\(R\ge1\). Let \(\chi_{\rm dy}\in C_c^\infty((-2,2))\) be even and equal
		to \(1\) on \([-1,1]\), and define
		\[
		\psi_0(\xi):=\chi_{\rm dy}(\xi),\qquad
		\psi_j(\xi):=\chi_{\rm dy}(2^{-j}\xi)
		-\chi_{\rm dy}(2^{1-j}\xi),\quad j\ge1.
		\]
		The dyadic Fourier blocks are
		\[
		\Delta_jg(x):=\sum_{k\in\mathbb Z}
		\psi_j(\xi_k)\widehat g(k)\mathrm e^{i\xi_kx},
		\qquad j\ge0.
		\]
		Choose an even function
		\(\widetilde\psi\in C_c^\infty(\{\eta:1/4<|\eta|<4\})\) that equals
		\(1\) on $\operatorname{supp}
		\bigl(\chi_{\rm dy}(\eta)-\chi_{\rm dy}(2\eta)\bigr),$ and let $\widetilde\psi_j(\xi):=\widetilde\psi(2^{-j}\xi), j\ge1.$ 
		Choose also \(\widetilde\psi_0\in C_c^\infty(\mathbb R)\) equal to \(1\)
		on \(\operatorname{supp}\psi_0\). Define
		\[
		m_{R,j}(0):=0,\qquad
		m_{R,j}(\xi):=\widetilde\psi_j(\xi)
		\mathrm e^{-R^2/\xi^2},\quad \xi\ne0,\qquad
		I_j:=\mathcal M_{m_{R,j}}\Delta_jf_{\rm e}.
		\]
		The identity \(\widetilde\psi_j=1\) on
		\(\operatorname{supp}\psi_j\), together with
		\(\sum_{j\ge0}\psi_j(\xi)=1\), gives the explicit decomposition
		\[
		(S(t)f)_{\rm e}=\sum_{j\ge0}I_j.
		\]
		
		For \(j\ge1\), the change of variables \(\xi=2^j\eta\) transforms the
		real-line multiplier \(m_{R,j}\) into
		\[
		\widetilde\psi(\eta)
		\exp\left(-\frac{(R/2^j)^2}{\eta^2}\right).
		\]
		On the support of \(\widetilde\psi\), \(|\eta|\) is bounded above and
		away from zero. For every \(M>0\),
		\[
		\sum_{\ell=0}^2
		\left\|
		\partial_\eta^\ell
		\left[
		\widetilde\psi(\eta)
		\exp\left(-\frac{(R/2^j)^2}{\eta^2}\right)
		\right]
		\right\|_2
		\le
		C_M\min\left\{1,\left(\frac{2^j}{R}\right)^M\right\}.
		\]
		When \(2^j<R\), this follows from
		\[
		u^q\mathrm e^{-cu^2}\le C_{q,M}u^{-M},
		\qquad u=\frac{R}{2^j}\ge1,
		\]
		for every fixed \(q\ge0\). When \(2^j\ge R\), one has
		\(R/2^j\le1\), and all derivatives in the preceding display are
		uniformly bounded. Taking \(M=1/2\), applying the \(H^1\)-to-\(L^1\)
		Fourier estimate proved above, rescaling, and periodizing give
		\[
		\|\mathcal M_{m_{R,j}}\|_{1\to1}
		\lesssim \min\left\{1,\left(\frac{2^j}{R}\right)^{1/2}\right\},
		\qquad j\ge1.
		\]
		For \(j=0\), only finitely many nonzero frequencies
		\(\xi_k\in\operatorname{supp}\widetilde\psi_0\) occur. Hence
		\[
		\|\mathcal M_{m_{R,0}}\|_{1\to1}
		\lesssim e^{-cR^2}
		\lesssim R^{-1/2}.
		\]
		Combining the two cases yields
		\[
		\|I_j\|_{L^1(\mathbb C_4)}
		\lesssim
		\min\left\{1,\left(\frac{2^j}{R}\right)^{1/2}\right\}
		\|\Delta_jf_{\rm e}\|_{L^1(\mathbb C_4)},
		\qquad j\ge0.
		\]
		Summing the terms \(I_j\) over \(j\ge0\) gives
		\[
		\begin{aligned}
			\|S(t)f\|_{L^1(I)}
			&\le \sum_{j\ge0}\|I_j\|_{L^1(\mathbb C_4)}\\
			&\lesssim R^{-1/2}
			\sum_{j\ge0}2^{j/2}
			\|\Delta_jf_{\rm e}\|_{L^1(\mathbb C_4)}\\
			&\lesssim R^{-1/2}
			\|f_{\rm e}\|_{B^{1/2}_{1,1}(\mathbb C_4)}
			=t^{-1/4}
			\|f_{\rm e}\|_{B^{1/2}_{1,1}(\mathbb C_4)}.
		\end{aligned}
		\]
		The dyadic sum is equivalent to the \(B^{1/2}_{1,1}\)-norm, so this summation
		introduces no logarithmic factor. This proves
		\eqref{eq:epd-Besov-highpass}.
	\end{proof}

	\subsection{Semigroup Estimates and Omitted Modes}
	
	\begin{lemma}\label{prop:epd-mainterms}
		For every \(N\ge2\) and \(t\ge0\),
		\begin{align*}
			\norm{P_NS(t)}_{L^\infty(I)\to L^\infty(I)}
			&\lesssim \Lambda_N\sqrt N\left(1+\frac{t}{N^2}\right)^{-1/4},\\
			\norm{JP_NS(t)}_{L^\infty(I)\to L^\infty(I)}
			&\lesssim \Lambda_N(1+t)^{-1/2},\\
			\norm{P_NS(t)J^*}_{L^\infty(I)\to L^\infty(I)}
			&\lesssim \Lambda_N\sqrt N(1+t)^{-1/2},\\
			\norm{JP_NS(t)J^*}_{L^\infty(I)\to L^\infty(I)}
			&\lesssim \Lambda_N(1+t)^{-1}\left(1+\frac{t}{N^2}\right)^{1/4}.
		\end{align*}
	\end{lemma}

	\begin{proof}
		Recall that $JP_N=\Pi_N J$.
		We first estimate \(P_NS(t)\). Equations \eqref{eq:epd-S0} and \eqref{eq:epd-weighted-Lebesgue} give
		\[
		\|P_NS(t)\|_{\infty\to\infty}
		\le \|P_N\|_{\infty\to\infty}\|S(t)\|_{\infty\to\infty}
		\lesssim \Lambda_N\sqrt N,\qquad t\ge0.
		\]
		For \(t\ge1\), applying Lemma~\ref{lem:epd-fractional} to the polynomial $k_N(x,y)$ in the variable $y$, together with \eqref{eq:epd-weighted-Lebesgue}, gives
		\[
		\begin{aligned}
			\sup_{x\in I}\|(k_N(x,\cdot))_{\rm e}\|_{B^{1/2}_{1,1}(\mathbb C_4)}
			&\lesssim \sqrt N\sup_{x\in I}
			\int_I |k_N(x,y)|(1-y^2)^{-1/4}\,\dd y\\
			&\lesssim \sqrt N\bigl(\sqrt N \Lambda_N\bigr)
			= N\Lambda_N.
		\end{aligned}
		\]
		Using the self-adjointness of \(S(t)\) and \eqref{eq:epd-Besov-highpass}, we obtain
		\[
		\begin{aligned}
			\|P_NS(t)\|_{\infty\to\infty}
			&=\sup_{x\in I}\|S(t)k_N(x,\cdot)\|_{L^1(I)}\\
			&\lesssim t^{-1/4}\sup_{x\in I}
			\|(k_N(x,\cdot))_{\rm e}\|_{B^{1/2}_{1,1}(\mathbb C_4)}
			\lesssim N\Lambda_Nt^{-1/4}.
		\end{aligned}
		\]
		Combining this estimate with the uniform bound gives, for \(t\ge1\),
		\[
		\begin{aligned}
			\|P_NS(t)\|_{\infty\to\infty}
			&\lesssim \Lambda_N\min\{\sqrt N,Nt^{-1/4}\}\\
			&=\Lambda_N\sqrt N\min\left\{1,\left(\frac{t}{N^2}\right)^{-1/4}\right\}\\
			&\lesssim \Lambda_N\sqrt N\left(1+\frac{t}{N^2}\right)^{-1/4}.
		\end{aligned}
		\]
		For \(0\le t\le1\), one has \(1\le F_N(t)\le2\), so the uniform bound gives the same estimate with a larger implicit constant.
		
		We next use \(JP_N=\Pi_N J\), \eqref{eq:epd-weighted-Lebesgue}, and \eqref{eq:epd-S1}:
		\[
		\begin{aligned}
			\|JP_NS(t)\|_{\infty\to\infty}
			&=\|\Pi_N JS(t)\|_{\infty\to\infty}\\
			&\le \|\Pi_N\|_{\infty\to\infty}\|JS(t)\|_{\infty\to\infty}
			\lesssim \Lambda_N(1+t)^{-1/2}.
		\end{aligned}
		\]
		The factor \(\Lambda_N\) here comes from \(\|\Pi_N\|_{\infty\to\infty}\lesssim \Lambda_N\).
		
		The semigroup identity \(S(t)=S(t/2)S(t/2)\) gives
		\[
		P_NS(t)J^*=P_NS(t/2)S(t/2)J^*.
		\]
		Using the estimate for \(P_NS(t/2)\), \eqref{eq:epd-S1} and $F_N(t/2)\ge1$, we obtain
		\[
		\begin{aligned}
			\|P_NS(t)J^*\|_{\infty\to\infty}
			&\le \|P_NS(t/2)\|_{\infty\to\infty}
			\|S(t/2)J^*\|_{\infty\to\infty}\\
			&\lesssim \Lambda_N\sqrt N\,F_N(t/2)^{-1/4}(1+t/2)^{-1/2}\\
			&\lesssim \Lambda_N\sqrt N\,(1+t)^{-1/2}.
		\end{aligned}
		\]
		
		Finally, \(JP_N=\Pi_N J\), \eqref{eq:epd-weighted-Lebesgue}, and \eqref{eq:epd-S2} give
		\[
		\begin{aligned}
			\|JP_NS(t)J^*\|_{\infty\to\infty}
			&=\|\Pi_N JS(t)J^*\|_{\infty\to\infty}\\
			&\le \|\Pi_N\|_{\infty\to\infty}
			\|JS(t)J^*\|_{\infty\to\infty}\\
			&\lesssim \Lambda_N(1+t)^{-1}\\
			&\lesssim \Lambda_N(1+t)^{-1}F_N(t)^{1/4}.
		\end{aligned}
		\]
		The last inequality uses $F_N(t)\ge1$.
	\end{proof}

	Only the first two omitted Legendre modes, $m=N,N+1$, occur in  \eqref{eq:epd-ranktwo}.

	\begin{lemma}\label{lem:epd-exterior}
		For $m\ge1$ and $t\ge0$,
		\begin{align}
			\norm{S(t)\ell_m}_1&\lesssim (1+t/m^2)^{-3/4},\label{eq:epd-B1}\\
			\norm{JS(t)\ell_m}_1&\lesssim m^{-1}(1+t/m^2)^{-5/4}.\label{eq:epd-B0}
		\end{align}
	\end{lemma}
	\begin{proof}
		Let $R=\sqrt t$. To prove \eqref{eq:epd-B1}, we first consider $R\le C_0m$, where $C_0>1$ is a fixed sufficiently large constant. Then \eqref{eq:epd-S0} and \(\norm{\ell_m}_1\lesssim 1\) give
		\[
		\norm{S(t)\ell_m}_1
		\le \norm{S(t)}_{1\to1}\norm{\ell_m}_1
		\lesssim (1+t/m^2)^{-3/4}.
		\]	
		We next consider \(C_0m<R<m^2\). Let $a:=\frac{m}{R}$. Then \(0<a<C_0^{-1}\) and \(ma>1\). For $x=\cos\theta$, choose a smooth cutoff function $\chi_a\in C^\infty([0,\pi])$ satisfying
		\[
		0\le\chi_a\le1,\qquad
		\chi_a(\theta)=1\quad\text{for }\theta\in[0,a]\cup[\pi-a,\pi],
		\]
		\[
		\chi_a(\theta)=0\quad\text{for }\theta\in[2a,\pi-2a],
		\qquad
		|\partial_\theta^j\chi_a(\theta)|\lesssim a^{-j},\quad j=1,2.
		\]
		Decompose $\ell_m=f_b+f_i$, where
		\[
		f_b(\cos\theta):=\chi_a(\theta)\ell_m(\cos\theta),\qquad
		f_i(\cos\theta):=(1-\chi_a(\theta))\ell_m(\cos\theta).
		\]
		Although \(f_b\) and \(f_i\) need not have zero mean, the extended definition of \(S(t)\) applies to each. The function $f_b$ is supported near the two endpoints. By \eqref{eq:bernstein} and \(dx=\sin\theta\,\dd\theta\),
		\[
		\begin{aligned}
			\norm{f_b}_1
			&\lesssim \int_0^{2a}
			\min\{m^{1/2},\theta^{-1/2}\}\theta\,\dd\theta\\
			&\lesssim \left(m^{-3/2}+a^{3/2}\right)
			\lesssim a^{3/2},
		\end{aligned}
		\]
		where \(a\ge m^{-1}\) was used in the last inequality.
		
		To estimate \(f_i\), define \(\widetilde\chi_a(\cos\theta)=\chi_a(\theta)\). 
		Using $|\widetilde\chi_a'(x)|\lesssim a^{-2},|\widetilde\chi_a''(x)|\lesssim a^{-4},$ together with \eqref{eq:bernstein}, \eqref{eq:epd-Leg-derivatives}, and \(dx=\sin\theta\,\dd\theta\), we obtain
		\[
		\norm{f_i''}_1
		\lesssim \left(m^2a^{-1/2}+ma^{-3/2}+a^{-5/2}\right).
		\]
		Together with \(ma\ge1\), this implies $\norm{f_i''}_1\lesssim m^2a^{-1/2}.$
		
		Let \(f_i^e\) be the even reflection of \(f_i\) at \(-1\) and \(1\), followed by its period-\(4\) extension to $\mathbb C_4.$ 
		The function \(f_i\) vanishes near both endpoints, so \((f_i^e)''\) has no boundary Dirac masses and
		\[
		\norm{(f_i^e)''}_{\mathrm{meas}}\lesssim \norm{f_i''}_1.
		\]
		Applying \eqref{eq:epd-S0} to \(f_b\) and \eqref{eq:epd-second-highpass} to \(f_i\), and using $a=m/R$, we obtain
		\[
		\begin{aligned}
			\norm{S(t)\ell_m}_1
			&\le \norm{S(t)f_b}_1+\norm{S(t)f_i}_1\\
			&\lesssim \norm{f_b}_1+t^{-1}\norm{(f_i^e)''}_{\mathrm{meas}}\\
			&\lesssim \left(a^{3/2}+R^{-2}m^2a^{-1/2}\right)
			\lesssim\left(\frac{m}{R}\right)^{3/2}.
		\end{aligned}
		\]
		
		Let \(\ell_m^e\) denote the even reflection and period-\(4\) extension of \(\ell_m\). Using \eqref{eq:epd-Leg-derivatives}, its distributional second derivative satisfies
		\[
		\norm{(\ell_m^e)''}_{\mathrm{meas}}
		\lesssim \left(\norm{\ell_m''}_1
		+|\ell_m'(1)|+|\ell_m'(-1)|\right)
		\lesssim m^{5/2}.
		\]
		We next use \eqref{eq:epd-second-highpass} and $R\ge m^2$ to obtain
		\[
		\norm{S(t)\ell_m}_1
		\lesssim t^{-1}\norm{(\ell_m^e)''}_{\mathrm{meas}}
		\lesssim m^{5/2}R^{-2}
		\lesssim \left(\frac{m}{R}\right)^{3/2}.
		\]
		Combining the three ranges proves \eqref{eq:epd-B1}.
		
		To prove \eqref{eq:epd-B0}, first suppose $R\le C_0m$.
		Using $JS(t)=S_D(t)J$ as established above, \eqref{eq:epd-S0}, $J\ell_m=-\eta_{m+1}$, and Lemma~\ref{lem:epd-Jacobi}, we obtain
		\[
		\norm{JS(t)\ell_m}_1
		=\norm{S_D(t)J\ell_m}_1
		\lesssim \norm{J\ell_m}_1
		=\norm{\eta_{m+1}}_1
		\lesssim \frac{1}{m}.
		\]
		Using \(R/m\le C_0\), we obtain \eqref{eq:epd-B0}.
		
		For $R>C_0m$, the semigroup property, \eqref{eq:epd-S1}, and \eqref{eq:epd-B1} give
		\[
		\begin{aligned}
			\norm{JS(t)\ell_m}_1
			&\le \norm{JS(t/2)}_{1\to1}\norm{S(t/2)\ell_m}_1\\
			&\lesssim (1+t/2)^{-1/2}
			\left(1+\frac{t}{2m^2}\right)^{-3/4}\\
			&\lesssim R^{-1}\left(\frac{m}{R}\right)^{3/2}=m^{-1}\left(\frac{R}{m}\right)^{-5/2}.
		\end{aligned}
		\]
		Combining the two ranges and using \(R^2=t\) proves \eqref{eq:epd-B0}, completing the proof. 
	\end{proof}

	\subsection{Heat Estimates on the Parity Chains}
	Split $M$ into the chains of even and odd indices. Fix a nonempty parity chain and let $k_{\mathrm{end}}\in\{N-1,N\}$ denote its last index. Define the node weights $\omega_k$, edge lengths $h_k$, and one-dimensional coordinates $s_k$ along this chain by
	\begin{equation}\label{eq:epd-path-weights}
		\omega_k=2k-1,\qquad h_k=2k+1,\qquad
		s_k=\frac{(k-1/2)^2}{2}.
	\end{equation}
	Here $h_k=s_{k+2}-s_k$, and
	$\omega_k=(h_{k-2}+h_k)/2$. Introduce a ghost vertex with zero Dirichlet value at $k=0$ for the
	even chain or $k=1$ for the odd chain, and at $k_{\mathrm{end}}+2$ on the right.
	The path length is comparable to $N^2$, and its largest edge is at most
	$CN$.

	Define the set of interior vertices of the chosen parity chain as follows
	\[
	\mathcal I=
	\begin{cases}
		\{2,4,\ldots,k_{\mathrm{end}}\}, & k_{\mathrm{end}}\ \text{is even},\\
		\{3,5,\ldots,k_{\mathrm{end}}\}, & k_{\mathrm{end}}\ \text{is odd}.
	\end{cases}
	\]
	For \(1\le p<\infty\), define the space of vectors on the interior vertices in $\mathcal I$ by
	\[
	\begin{aligned}
		\ell^p(\omega):&=\left\{u=(u_k)_{k\in\mathcal I}:
		\|u\|_{\ell^p(\omega)}:=\left(\sum_{k\in\mathcal I}\omega_k|u_k|^p\right)^{1/p}
		<\infty
		\right\},\\
		\ell^\infty(\omega):&=\left\{u=(u_k)_{k\in\mathcal I}:
		\|u\|_{\ell^\infty(\omega)}:=\max_{k\in\mathcal I}|u_k|
		<\infty\right\}.
	\end{aligned}
	\]
	Define the weighted inner product by $\ip{u}{v}_\omega:=\sum_k\omega_k u_k\overline{v_k}.$
	
	Define the transformed operator $\mathcal L$ by
	\begin{equation}\label{eq:epd-pathL}
		(\mathcal L u)_k=\frac1{\omega_k}
		\left(\frac{u_k-u_{k-2}}{h_{k-2}}
		+\frac{u_k-u_{k+2}}{h_k}\right).
	\end{equation}
	On the chosen parity chain, its relation to the mass matrix is
	\[
	(\mathcal Lu)_k=\frac{(Ma)_k}{\sqrt{\omega_k}},\qquad a_k=\sqrt{\omega_k}u_k.
	\]
	It is a positive Dirichlet finite-volume Laplacian, self-adjoint in
	$\ell^2(\omega)$. Let $p_t(k,j)$ denote its heat kernel with respect to
	the measure $\omega$, so that
	\[
	(\mathrm e^{-t\mathcal L}f)_k=\sum_j p_t(k,j)f_j\omega_j.
	\]
	
	\begin{lemma}\label{lem:epd-path-heat}
		For $t>0$, we have the estimates
		\begin{align}
			0\le p_t(k,j)&\lesssim t^{-1/2} ,\label{eq:epd-Nashheat}\\
			q_{k_{\mathrm{end}}}(t):=\sum_k\omega_kp_t(k,k_{\mathrm{end}})
			&\lesssim \min\{1,N/\sqrt t\},\label{eq:epd-survival}\\
			p_{R^2}(k,j)&\lesssim \frac 1 R\exp\left(-\frac{c|s_k-s_j|}{R}\right)
			\quad(R\ge N).\label{eq:epd-Davies}\\
			\norm{\bm1_{k\le N/2}\mathrm e^{-tM}e_{k_{\mathrm{end}}}}_{\ell^2}
			&\lesssim \mathrm e^{-cN},
			\qquad 0\le t\le N^2.\label{eq:epd-short-block}
		\end{align}
	\end{lemma}
	\begin{proof}
		
		The operator \(\mathcal L\) in \eqref{eq:epd-pathL} has nonpositive off-diagonal coefficients and Dirichlet conditions at the two ghost vertices. Hence \(\mathrm e^{-t\mathcal L}\) is positivity preserving and sub-Markovian.
		
		For a nodal vector \(u=(u_k)\), let \(\widetilde u\) be its continuous
		piecewise-linear interpolant in the coordinate \(s\), with value zero at
		the two ghost vertices and zero outside the path. The weighted nodal norms
		and the continuous norms satisfy
		\[
		\norm{u}_{\ell^p(\omega)}\lesssim\norm{\widetilde u}_{L^p(\R)}
		\lesssim \norm{u}_{\ell^p(\omega)},\qquad p=1,2.
		\]
		Define the discrete Dirichlet energy by
		\[
		\mathcal E(u):=
		\sum_{\text{edges }(k,k+2)}
		\frac{|u_{k+2}-u_k|^2}{h_k},
		\]
		Piecewise linearity gives $\mathcal E(u)=\int_{\R}|\widetilde u'(s)|^2\,\dd s.$ 
		For \(a>0\), splitting the Fourier integral into \(\{|\xi|\le a\}\) and \(\{|\xi|>a\}\) gives
		\[
		\norm{\widetilde u}_{L^2(\R)}^2
		\lesssim \left(
		a\norm{\widetilde u}_{L^1(\R)}^2
		+a^{-2}\norm{\widetilde u'}_{L^2(\R)}^2
		\right).
		\]
		Optimizing in \(a\) and using the norm equivalences above yields the discrete Nash inequality
		\[
		\norm{u}_{\ell^2(\omega)}^6
		\lesssim \mathcal E(u)\norm{u}_{\ell^1(\omega)}^4.
		\]
		
		Let \(u(t)=\mathrm e^{-t\mathcal L}f\) and
		$y(t):=\norm{u(t)}_{\ell^2(\omega)}^2$.
		Self-adjointness and the sub-Markov property give
		\[
		y'(t)=-2\mathcal E(u(t)),
		\qquad
		\norm{u(t)}_{\ell^1(\omega)}\le\norm{f}_{\ell^1(\omega)}.
		\]
		The Nash inequality therefore implies
		\[
		y'(t)=-2\mathcal E(u(t))
		\le -\frac{2\norm{u(t)}_{\ell^2(\omega)}^6}
		{C\norm{u(t)}_{\ell^1(\omega)}^4}
		\le -\frac{c\,y(t)^3}{\norm{f}_{\ell^1(\omega)}^4}.
		\]
		Integration gives
		$\norm{\mathrm e^{-t\mathcal L}f}_{\ell^2(\omega)}
		\lesssim t^{-1/4}\norm{f}_{\ell^1(\omega)}$.
		Self-adjointness gives the corresponding \(2\to\infty\) estimate; together
		with the semigroup identity, this gives
		\[
		\norm{\mathrm e^{-t\mathcal L}}_{1\to\infty}
		\le
		\norm{\mathrm e^{-(t/2)\mathcal L}}_{1\to2}
		\norm{\mathrm e^{-(t/2)\mathcal L}}_{2\to\infty}
		\lesssim t^{-1/2}.
		\]
		The weighted kernel of \(\mathrm e^{-t\mathcal L}\) is \(p_t(k,j)\), so positivity and the last estimate give $0\le p_t(k,j)\lesssim t^{-1/2},$ which proves \eqref{eq:epd-Nashheat}.
		
		We next estimate $q_{k_{\mathrm{end}}}(t)
		:=\sum_k\omega_kp_t(k,k_{\mathrm{end}})
		=(\mathrm e^{-t\mathcal L}\bm1)_{k_{\mathrm{end}}},$
		where \(\bm1=\mathbf{1}_{\mathcal I}\) denotes the indicator vector of \(\mathcal I\). Let $\delta_k:=s_{k_{\mathrm{end}}+2}-s_k$ be the \(s\)-distance from \(k\) to the right ghost vertex. For \(\lambda>0\), define
		\[
		U:=(\mathcal L+\lambda)^{-1}\bm1,
		\qquad
		R:=\lambda^{-1/2},
		\qquad
		B_k:=R^2\left(1-\mathrm e^{-\delta_k/R}\right).
		\]
		A direct calculation gives
		\[
		(\mathcal LB)_k\geq \mathrm e^{-\delta_k/R}\frac{R}{\omega_k}
		\left(
		\frac{\mathrm e^{h_k/R}-1}{h_k/R}-\frac{1-\mathrm e^{-h_{k-2}/R}}{h_{k-2}/R}
		\right)\ge\frac16\mathrm e^{-\delta_k/R}.
		\]
		Together with $\lambda B_k=1-\mathrm e^{-\delta_k/R},$ we obtain $(\mathcal L+\lambda) B\ge\frac16\bm1.$
		Since \(B\) is nonnegative at the left ghost vertex, the discrete maximum principle gives
		\begin{equation}\label{eq:Uk-bound}
			0\le U_k\le6B_k
			\lesssim \min\{R^2,R\delta_k\}.
		\end{equation}
		The resolvent formula gives
		\[
		U_{k_{\mathrm{end}}}=\int_0^\infty
		\mathrm e^{-\lambda r}q_{k_{\mathrm{end}}}(r)\,\dd r.
		\]
		For \(t>0\), choose \(\lambda=t^{-1}\). Monotonicity of \(q_{k_{\mathrm{end}}}\) yields
		\begin{equation}\label{eq:qkenfd-leq-U}
			U_{k_{\mathrm{end}}}
			\ge
			\int_0^t\mathrm e^{-r/t}q_{k_{\mathrm{end}}}(t)\,\dd r=t(1-\mathrm e^{-1})q_{k_{\mathrm{end}}}(t).
		\end{equation}
		Combining \eqref{eq:Uk-bound}, \eqref{eq:qkenfd-leq-U} and using $\delta_{k_{\mathrm{end}}} =h_{k_{\mathrm{end}}}\asymp N$, we obtain \eqref{eq:epd-survival}.
		
		It remains to prove the spatial localization estimate. For a real nodal vector \(\psi=(\psi_k)\), let \(\mathrm e^\psi\) denote diagonal multiplication by \(\mathrm e^{\psi_k}\). Fix \(R\ge N\), choose $\psi_k=\pm\frac{s_k}{10R},$ and define $T_\psi(t):=\mathrm e^\psi\mathrm e^{-t\mathcal L}\mathrm e^{-\psi}.$
		
		Assign the same formula $\psi_k=\pm s_k/(10R)$ at the two ghost vertices, while retaining zero Dirichlet values for the input vectors. The coordinate increments satisfy
		\[
		\frac{\psi_k-\psi_{k-2}}{h_{k-2}}+\frac{\psi_k-\psi_{k+2}}{h_k}=0.
		\]
		Since the ghost vertices contribute no off-diagonal terms to the Dirichlet operator, we have
		\[
		\begin{aligned}
			\left(\mathrm e^\psi\mathcal L\mathrm e^{-\psi}\mathbf1\right)_k
			&\ge\frac1{\omega_k}\left(
			\frac{1-\mathrm e^{\psi_k-\psi_{k-2}}}{h_{k-2}}
			+\frac{1-\mathrm e^{\psi_k-\psi_{k+2}}}{h_k}
			\right)\\
			&\ge-\frac{C(h_{k-2}+h_k)}{\omega_kR^2}
			\ge-CR^{-2}.
		\end{aligned}
		\]
		Here Taylor's formula, $h_k\lesssim N$, $R\ge N$, and $h_{k-2}+h_k=2\omega_k$ control the quadratic remainders. Applying the same estimate with $-\psi$, positivity and weighted duality give the $1\to1$ and $\infty\to\infty$ bounds below. For the quadratic-form estimate, conjugation leaves the boundary terms unchanged, while the interior-edge corrections are controlled by
		\[
		\cosh(\psi_{k+2}-\psi_k)-1\lesssim h_k^2/R^2.
		\]
		Consequently, we have
		\[
		\begin{aligned}
			\norm{T_\psi(t)}_{1\to1}\le\mathrm e^{Ct/R^2},
			\qquad
			\norm{T_\psi(t)}_{\infty\to\infty}\le\mathrm e^{Ct/R^2},\\
			\operatorname{Re}
			\ip{\mathrm e^\psi\mathcal L\mathrm e^{-\psi}v}{v}_\omega
			\ge
			\mathcal E(v)-CR^{-2}\norm{v}_{\ell^2(\omega)}^2.
		\end{aligned}
		\]
		Applying the preceding Nash argument to \(T_\psi(t)\), and using the semigroup property at $t/2$ and $T_\psi(t)^*=T_{-\psi}(t),$ gives
		\[
		\norm{T_\psi(t)}_{1\to\infty}\leq \norm{T_\psi(t/2)}_{1\to2} \norm{T_\psi(t/2)}_{2\to\infty} 
		\lesssim t^{-1/2}\mathrm e^{Ct/R^2}.
		\]
		The weighted kernel of \(T_\psi(t)\) is $\mathrm e^{\psi_k-\psi_j}p_t(k,j),$ so we have
		\[
		p_t(k,j) \lesssim t^{-1/2}
		\exp\left(\frac{Ct}{R^2}-\psi_k+\psi_j\right).
		\]
		Choose the sign of \(\psi\) such that $\psi_k-\psi_j=\frac{|s_k-s_j|}{10R},$
		and take \(t=R^2\). Absorbing the fixed factor \(\mathrm e^C\) into the implicit constant gives
		\[
		p_{R^2}(k,j)
		\lesssim \frac{1}{R}
		\exp\left(-\frac{c|s_k-s_j|}{R}\right),
		\]
		which proves \eqref{eq:epd-Davies}.
		
		The quadratic-form estimate also gives the short-time block bound. Under the unitary change of coordinates \(a_k=\sqrt{\omega_k}u_k\), the operator \(\mathcal L\) is represented by \(M\), and diagonal multiplication by \(\mathrm e^\psi\) is unchanged. 
		Take $R=N$ and $\psi_k=-s_k/(10N)$. Define $s_{N/2}=\frac12\left(\frac N2-\frac12\right)^2$. For every vertex with $k\le N/2$, we have $s_k\le s_{N/2}$ and hence
		\[
		s_{k_{\mathrm{end}}}-s_k
		\ge s_{k_{\mathrm{end}}}-s_{N/2}.
		\]
		Therefore, for $0\le t\le N^2$,
		\[
		\norm{\bm1_{k\le N/2}\mathrm e^{-tM}e_{k_{\mathrm{end}}}}_{\ell^2}
		\le
		\exp\left(
		\frac{Ct}{N^2}
		-\frac{s_{k_{\mathrm{end}}}-s_{N/2}}{10N}
		\right),
		\]
		where \(\bm1_{k\le N/2}\) is the coordinate projection onto the vertices satisfying \(k\le N/2\), and \(e_{k_{\mathrm{end}}}\) is the Euclidean coordinate vector at \(k_{\mathrm{end}}\). Using $s_{k_{\mathrm{end}}}-s_{N/2}\ge cN^2,$ and absorbing the fixed factor \(\mathrm e^C\) into \(C\), we obtain \eqref{eq:epd-short-block}
	\end{proof}
	
	\begin{lemma}\label{lem:epd-unimodal}
		For $k\in\mathcal I$ and fixed $t>0$ the sequence
		\[
		u_k(t)=\frac{(\mathrm e^{-tM}e_{k_{\mathrm{end}}})_k}{\sqrt{\omega_k}}
		=\sqrt{\omega_{k_{\mathrm{end}}}}\,p_t(k,k_{\mathrm{end}})
		\]
		is positive and unimodal along its parity chain, with zero ghost values.
		The total variation of the sequence satisfies
		\[
		\sum_{k\in\mathcal I}|u_k(t)-u_{k-2}(t)|
		+|u_{k_{\mathrm{end}}+2}(t)-u_{k_{\mathrm{end}}}(t)|
		\le 2\max_{k\in\mathcal I}u_k(t).
		\]
		\[
		\sum_{\substack{k\in\mathcal I\\ k\le m}}|u_k(t)-u_{k-2}(t)|
		\le 2\max_{\substack{k\in\mathcal I\\ k\le m}}u_k(t),
		\qquad m\in\mathcal I.
		\]
	\end{lemma}
	\begin{proof}
		We prove unimodality by showing that the edge differences change sign exactly once, from positive to negative. Relabel the interior vertices of the chosen parity chain by $i=1,\ldots,n$, and renumber their weights and edge lengths accordingly. Set $u_0(t)=u_{n+1}(t)=0$ at the two ghost vertices. For $0\le i\le n$, let $c_i=h_i^{-1}$ denote the conductance of the edge $(i,i+1)$ and define its difference by $v_i(t):=u_{i+1}(t)-u_i(t)$. The heat equation reads
		\[
		\dot u_i=\frac{c_i v_i-c_{i-1}v_{i-1}}{\omega_i},
		\qquad 1\le i\le n.
		\]
		Taking adjacent differences gives
		\begin{align*}
			\dot v_0&=-\frac{c_0}{\omega_1}v_0+\frac{c_1}{\omega_1}v_1,\\
			\dot v_i&=\frac{c_{i-1}}{\omega_i}v_{i-1}
			-c_i\left(\frac1{\omega_i}+\frac1{\omega_{i+1}}\right)v_i
			+\frac{c_{i+1}}{\omega_{i+1}}v_{i+1},
			\qquad 1\le i\le n-1,\\
			\dot v_n&=\frac{c_{n-1}}{\omega_n}v_{n-1}-\frac{c_n}{\omega_n}v_n.
		\end{align*}
		Hence the edge-difference vector $v(t):=(v_0(t),\ldots,v_n(t))^{\top}$ satisfies $\dot v=Av$, where $A$ is tridiagonal and $A_{ij}\ge0$ for $i\ne j$.
		
		We count sign changes after deleting zero entries. Let $I_{n+1}$ be the identity matrix on $\mathbb R^{n+1}$ and let $E_{ij}$ have a single $1$ in position $(i,j)$ and zeros elsewhere. The decomposition
		\[
		A=\operatorname{diag}(A_{00},\ldots,A_{nn})
		+\sum_{\substack{0\le i,j\le n\\ |i-j|=1}}A_{ij}E_{ij}
		\]
		shows why the evolution cannot increase the number of sign changes. The exponential of the diagonal part multiplies each coordinate by a positive number, while
		\[
		e^{sA_{ij}E_{ij}}=I_{n+1}+sA_{ij}E_{ij},
		\qquad s\ge0,\quad |i-j|=1,
		\]
		adds a nonnegative multiple of one coordinate to its neighbour. Neither operation increases the number of sign changes. By the Lie product formula, $v(t)=e^{tA}v(0)$ has no more sign changes than $v(0)$.
		
		The initial point source at the last interior vertex gives
		\[
		v_{n-1}(0)=\omega_n^{-1/2},\qquad
		v_n(0)=-\omega_n^{-1/2},\qquad
		v_i(0)=0\quad(0\le i<n-1),
		\]
		so $v(0)$ has exactly one sign change. For $t>0$, positivity of the heat flow on the connected chain gives $u_i(t)>0$ for $1\le i\le n$, and hence
		\[
		v_0(t)=u_1(t)>0,\qquad v_n(t)=-u_n(t)<0.
		\]
		The differences therefore change sign exactly once, from positive to negative, proving that $u(t)$ is unimodal.
		
		Finally, summing the increasing and decreasing parts of any initial segment yields
		\[
		\begin{aligned}
			\sum_{i=0}^{m-1}|u_{i+1}(t)-u_i(t)|
			&=2\max_{0\le i\le m}u_i(t)-u_m(t)\\
			&\le 2\max_{0\le i\le m}u_i(t),
			\qquad 1\le m\le n+1.
		\end{aligned}
		\]
		Taking $m=n+1$ proves the total variation bound including both ghosts; taking $m\le n$ proves the bound on each initial segment.
	\end{proof}

	\subsection{Estimates for Sources at the Terminal Mode Indices}
	For the terminal index $k_{\mathrm{end}}$ of either parity chain, define
	the following amplitudes for the function and its derivative:
	\[
	\mathcal A_a(t):=\sup_{x\in I}|z^a(x)^{\top}\mathrm e^{-tM}e_{k_{\mathrm{end}}}|,\qquad a=0,1.
	\]
	
	\begin{lemma}\label{prop:epd-Asources}
		Uniformly in $N$ and $t\ge0$,
		\begin{align}
			\mathcal A_1(t)&\lesssim \sqrt N(1+t/N^2)^{-1/2},\label{eq:epd-A1}\\
			\mathcal A_0(t)&\lesssim \sqrt{\Lambda_N}\,N^{-1}(1+t/N^2)^{-3/4}.\label{eq:epd-A0}
		\end{align}
	\end{lemma}
	\begin{proof}
		We estimate $\mathcal A_1$ using positivity and weighted mass decay. For $\mathcal A_0$, we combine heat localization with summation by parts to exploit cancellation in the basis expansion. 
		The vector $e^{-tM}e_{k_{\mathrm{end}}}$ has nonnegative entries. Since
		$|\ell_{k-1}(x)|\le\sqrt{\omega_k/2}$,
		\[
		\mathcal A_1(t)\le\frac1{\sqrt2}\sum_k\omega_k u_k(t)
		=\sqrt{\omega_{k_{\mathrm{end}}}/2}\,q_{k_{\mathrm{end}}}(t).
		\]
		Equation \eqref{eq:epd-survival} proves \eqref{eq:epd-A1}.
		
		First suppose $0\le t\le N^2$. On $k\ge N/2$,
		$\norm{\eta_k}_\infty\lesssim 1/N$, while
		\[
		\sum_{k\ge N/2}(\mathrm e^{-tM}e_{k_{\mathrm{end}}})_k
		\lesssim N^{-1/2}\sum_k\omega_k u_k(t)\lesssim 1.
		\]
		For the remaining indices use \eqref{eq:epd-short-block} and
		$\sup_x\sum_{k=2}^N\eta_k(x)^2\lesssim 1$.
		The latter follows from the norm of point evaluation in $H^1_0(I)$,
		because the $\eta_k$ are energy-orthonormal.
		Thus $\mathcal A_0(t)\lesssim 1/N$ in this range.
		
		For the remaining time range, let $R=\sqrt t\ge N$ denote the diffusion
		length. Equations
		\eqref{eq:epd-Nashheat} and \eqref{eq:epd-survival}, together with the semigroup property at $t/2$,
		give
		\begin{equation}\label{eq:epd-mass-height}
			U:=\sum_k\omega_k u_k(t)\lesssim \frac{N^{3/2}}R,
			\qquad H:=\max_k u_k(t)\lesssim \frac{N^{3/2}}{R^2}.
		\end{equation}
		The basis expansion gives
		\begin{equation}\label{eq:epd-Abel-synthesis}
			z^0(x)^{\top}\mathrm e^{-tM}e_{k_{\mathrm{end}}}
			=\frac1{\sqrt2}\sum_k u_k(t)\,[L_k(x)-L_{k-2}(x)].
		\end{equation}
		Since the ghost values vanish, summation by parts gives
		\[
		\sum_{k\in\mathcal I}u_k(t)\bigl[L_k(x)-L_{k-2}(x)\bigr]
		=
		u_{k_{\mathrm{end}}}(t)L_{k_{\mathrm{end}}}(x)
		-\sum_{k\in\mathcal I}
		\bigl[u_k(t)-u_{k-2}(t)\bigr]L_{k-2}(x).
		\]
		Using $|L_j(x)|\le1$ and Lemma 4.9, we obtain
		\[
		\left|
		\sum_{k\in\mathcal I}u_k(t)\bigl[L_k(x)-L_{k-2}(x)\bigr]
		\right|
		\le
		u_{k_{\mathrm{end}}}(t)
		+\sum_{k\in\mathcal I}|u_k(t)-u_{k-2}(t)|
		\le 2H.
		\]
		Hence $A_0(t)\le\sqrt{2}H$. 
		
		If $R\ge N^2/(B \Lambda_N)$, for a sufficiently large constant $B$,
		this already gives
		\[
		\mathcal A_0(t)\lesssim \frac{N^{3/2}}{R^2}
		\lesssim \frac{\sqrt{\Lambda_N}\sqrt N}{R^{3/2}}.
		\]
		
		It remains to consider $N\le R<N^2/(B \Lambda_N)$, a range which may be empty
		for small $N$. Split the sum at $k\asymp N/2$. On the low-index part,
		\eqref{eq:epd-Davies} gives
		\[
		\max_{k\le N/2}u_k(t)
		\lesssim \frac{\sqrt N}R\exp(-cN^2/R).
		\]
		Summation by parts on this initial segment, together with the variation bound for initial segments in Lemma~\ref{lem:epd-unimodal}, bounds its contribution by the same quantity. 		
		For sufficiently large $B$, this contribution is bounded by $C\sqrt N\,R^{-3/2}$, since $\sqrt R\exp(-cN^2/R)\lesssim1$ throughout this range.
		
		For the high-index part write $x=\cos\theta$ and $b=\sin\theta$.
		Using the difference bound in \eqref{eq:epd-leg-diff} and positivity gives
		\[
		\left|\sum_{k\ge N/2}u_k(L_k-L_{k-2})\right|
		\lesssim \frac{\sqrt b}{\sqrt N}\frac UN
		\lesssim \frac{\sqrt b}{R}.
		\]
		On the other hand, summation by parts, unimodality, and the Bernstein
		estimate \eqref{eq:bernstein}, with $L_k=\sqrt{2/(2k+1)}\,\ell_k$, give
		\[
		\left|\sum_{k\ge N/2}u_k(L_k-L_{k-2})\right|
		\lesssim \frac{N^{3/2}}{R^2}\min\{1,(Nb)^{-1/2}\}.
		\]
		For $b\ge N^{-1}$, taking the geometric mean of the two bounds gives an upper bound of $C\sqrt N\,R^{-3/2}$.
		For $b<N^{-1}$, the first estimate gives the bound $C(\sqrt N R)^{-1}$, which is at most $C\sqrt N\,R^{-3/2}$ because $R\le N^2$.
	\end{proof}
	\subsection{Completion of the Upper Bound}
	\begin{proof}[Proof of Theorem~\ref{thm:epd-main}]		
		
		%
		%

		We first establish the four bounds for $h_{ab}$, then estimate the time integral representing $\mathbb T_N$. For $a,b\in\{0,1\}$ and admissible $\nu$, recall the error kernels $E_{t,\nu}^{ab}$ defined in (4.5). We use the decomposition \eqref{eq:epd-kernel-split}, established in Section~4.1.
		
		Lemma~\ref{prop:epd-mainterms} controls the main kernel \(\mathcal K_t^{ab}\). We next estimate each error kernel \(E_{t,\nu}^{ab}\)
		Fix one admissible $\nu$, put $k_{\mathrm{end}}=N-\nu$ and $m=N+1-\nu$, and define $B_b(t):=\|D_bS(t)\ell_m\|_1.$
		Using $\rho_{N,\nu}\lesssim N^{-2}$ gives
		\begin{equation}\label{eq:epd-convbound}
			\sup_{x\in I}\int_I|E_{t,\nu}^{ab}(x,y)|\,\dd y
			\lesssim N^{-2}\int_0^t \mathcal A_a(t-s)B_b(s)\,\dd s.
		\end{equation}
		Lemmas~\ref{prop:epd-Asources} and~\ref{lem:epd-exterior} give bounds for $\mathcal A_a$ and $B_b$. 
		Set $T=t/N^2$. The substitution $s=N^2v$ cancels the factor $N^{-2}$ in \eqref{eq:epd-convbound}:
		\[
		N^{-2}\int_0^t \mathcal A_a(t-s)B_b(s)\,\dd s
		=\int_0^T \mathcal A_a(N^2(T-v))B_b(N^2v)\,\dd v.
		\]
		For $\alpha\in\{1/2,3/4\}$ and $\beta\in\{3/4,5/4\}$, splitting the integral at $T/2$ yields
		\[
		\int_0^T(1+T-v)^{-\alpha}(1+v)^{-\beta}\,\dd v\lesssim 
		\begin{cases}
			(1+T)^{1-\alpha-\beta},&\beta=3/4,\\
			(1+T)^{-\alpha},&\beta=5/4.
		\end{cases}
		\]
		Substituting the four source bounds therefore gives
		\begin{equation}\label{eq:epd-conv-table}
			\sup_{x\in I}\int_I|E_{t,\nu}^{ab}(x,y)|\,\dd y
			\lesssim 
			\begin{cases}
				\sqrt N\,F_N(t)^{-1/4},&(a,b)=(1,1),\\
				\sqrt{\Lambda_N}N^{-1}F_N(t)^{-1/2},&(a,b)=(0,1),\\
				N^{-1/2}F_N(t)^{-1/2},&(a,b)=(1,0),\\
				\sqrt{\Lambda_N}N^{-2}F_N(t)^{-3/4},&(a,b)=(0,0).
			\end{cases}
		\end{equation}
		Using $1+t\le N^2F_N(t)$ and $\sqrt{\Lambda_N}\le \Lambda_N$, combine \eqref{eq:epd-conv-table} with Lemma~\ref{prop:epd-mainterms} and sum over at most two values of $\nu$ to obtain \eqref{eq:epd-four-targets}. 
		
		It remains to estimate the integral in \eqref{eq:epd-route-time-integral}. Multiplying the bounds in \eqref{eq:epd-four-targets} gives
		\begin{equation}\label{eq:epd-product-target}
			h_{11}(t)h_{00}(t)+h_{10}(t)h_{01}(t)
			\lesssim \frac{\sqrt N\,\Lambda_N^2}{1+t}.
		\end{equation}
		To control the integral for large $t$, Lemma~\ref{lem:mass-spectrum} gives $\|\mathrm e^{-tM}\|_{2\to2}\le \mathrm e^{-ct/N^4}$ for some $c>0$ independent of $N$. The vectors of basis functions satisfy
		\[
		\begin{aligned}
			\sup_{x\in I}\|z^1(x)\|_2\lesssim N,&\qquad
			\int_I\|z^1(y)\|_2\,\dd y\lesssim \sqrt N,\\
			\sup_{x\in I}\|z^0(x)\|_2&+\int_I\|z^0(y)\|_2\,\dd y\lesssim 1.
		\end{aligned}
		\]
		Cauchy--Schwarz in the coefficient vectors therefore gives
		\[
		h_{ab}(t)
		\le \sup_{x\in I}\|z^a(x)\|_2
		\int_I\|z^b(y)\|_2\,\dd y\,
		\mathrm e^{-ct/N^4}
		\lesssim N^2\mathrm e^{-ct/N^4}.
		\]	
		Choose a fixed $A$ with $2cA\ge8$ and set $t_*=AN^4\Lambda_N$. Equation \eqref{eq:epd-product-target} gives
		\begin{equation}\label{eq:epd-integral-main}
			\int_0^{t_*}[h_{11}h_{00}+h_{10}h_{01}]\,\dd t
			\lesssim \sqrt N\,\Lambda_N^2\log(1+t_*)
			\lesssim \sqrt N\,\Lambda_N^3.
		\end{equation}
		Here $\log(1+t_*)\lesssim 1+4\log N+\log \Lambda_N\lesssim \Lambda_N$, so the time integral contributes the third factor $\Lambda_N$. The exponential bound controls the remaining integral:
		\[
		\begin{aligned}
			\int_{t_*}^\infty[h_{11}h_{00}+h_{10}h_{01}]\,\dd t
			&\lesssim N^4\int_{t_*}^\infty\mathrm e^{-2ct/N^4}\,\dd t\\
			&\lesssim N^8\mathrm e^{-2cA\Lambda_N}
			\lesssim N^8(N+1)^{-2cA}\lesssim 1.
		\end{aligned}
		\]
		Combining the two time intervals proves $\|\mathbb T_N\|_{\infty\to\infty}\lesssim \sqrt N\,\Lambda_N^3,$ completing the proof.
	\end{proof}
	
	\subsection{Elliptic Approximation}
	\label{sec:consequences}
	As a consequence of the stability estimates established above, we obtain the following approximation result. 
	\begin{corollary}\label{cor:ritz-approximation}
		For $u\in\mathcal W$,
		\begin{equation}\label{eq:ritz-best-infty}
			\norm{\grad(u-R_Nu)}_\infty
			\le (1+S_{N,\infty})E_N^{1,\infty}(u)
			\lesssim \sqrt N \Lambda_N^3 E_N^{1,\infty}(u).
		\end{equation}
		For every $v\in H^1_0(\Om)$,
		\begin{equation}\label{eq:ritz-best-two}
			\norm{\grad(v-R_Nv)}_2=E_N^{1,2}(v),\qquad
			\norm{v-R_Nv}_2\lesssim E_N^{1,2}(v).
		\end{equation}
	\end{corollary}
	\begin{proof}
		Idempotence gives $u-R_Nu=(\mathrm{Id}-R_N)(u-v_N)$ for every $v_N\in V_N$;
		apply Theorem~\ref{thm:unconditional} and take the infimum.
		Energy orthogonality gives
		\[
		\norm{\grad(v-v_N)}_2^2
		=\norm{\grad(v-R_Nv)}_2^2
		+\norm{\grad(R_Nv-v_N)}_2^2.
		\]
		Taking $v_N=R_Nv$ proves the equality in \eqref{eq:ritz-best-two};
		Poincar\'e's inequality proves its second assertion.
	\end{proof}
	
	If $u\in\mathcal W$ is the weak solution of $-\Delta u=f$ with zero
	boundary trace, the exact Galerkin solution is $R_Nu$.  
	Thus \eqref{eq:ritz-best-infty} is also the elliptic gradient error estimate.
	The Ritz projection depends only on $V_N$ and the unweighted Dirichlet-energy inner product. Hence the same projection and approximation bounds hold for Chebyshev, Gegenbauer, and Jacobi Galerkin formulations that use the same inner product and whose boundary-adapted tensor bases span $V_N$.

	
	\section{Parabolic Approximation}\label{sec:parabolic}
	
	Fix $T>0$ and consider
	\begin{equation}\label{eq:parabolic-eq}
		\begin{cases}
			u_t-\Delta u=f&\text{in }\Om\times(0,T),\\
			u=0&\text{on }\partial\Om\times(0,T),\\
			u(0)=u_0&\text{in }\Om.
		\end{cases}
	\end{equation}
	We assume $f\in L^2(0,T;L^2(\Om))$ and that its weak solution satisfies
	\begin{equation}\label{eq:parabolic-regularity}
		u\in C([0,T];\mathcal W)\cap W^{1,\infty}(0,T;H^1_0(\Om)).
	\end{equation}
	The weak formulation of \eqref{eq:parabolic-eq} is $(u_t(t),v)+(\grad u(t),\grad v)=(f(t),v)$ for every $v\in H^1_0(\Om)$ and almost every $t\in(0,T)$.  Assumption~\eqref{eq:parabolic-regularity}
	specifies the time regularity required below; spatial analyticity alone
	does not ensure this regularity at $t=0$.
	
	For $u_{N,0}\in V_N$, let $u_N\in H^1(0,T;V_N)$ solve
	\begin{equation}\label{eq:heat-galerkin}
		(u_{N,t},v_N)+(\grad u_N,\grad v_N)=(f,v_N),
		\quad v_N\in V_N,\qquad u_N(0)=u_{N,0}.
	\end{equation}
	The solution to \eqref{eq:heat-galerkin} exists and is unique.
	All estimates in this section concern spatial semidiscretization, with no time-stepping or quadrature error.
	
	\subsection{The Discrete Heat Operator and Error Equation}
	
	\begin{lemma}\label{lem:gradient-inverse}
		For every $v_N\in V_N$,
		\begin{equation*}
			\norm{\grad v_N}_\infty\lesssim N^2\norm{\grad v_N}_2.
		\end{equation*}
	\end{lemma}
	\begin{proof}
		For $q\in\Pp_N(I)\otimes\Pp_N(I)$, expand $q$ in the orthonormal
		basis $\ell_i(x)\ell_j(y)$, $0\le i,j\le N$.  Since
		$|L_j(x)|\le1$ on $I$,
		\begin{equation*}
			\sum_{j=0}^N|\ell_j(x)|^2\le\frac{(N+1)^2}{2}.
		\end{equation*}
		Cauchy--Schwarz in the coefficient array therefore gives
		$\norm{q}_\infty\le (N+1)^2\norm{q}_2/2$.
		Apply this bound to each component of $\grad v_N$ and sum their
		squares.  Both components belong to this tensor polynomial space.
	\end{proof}
	Let $\mathcal P_N $ be the scalar $L^2(\Om)$
	orthogonal projection onto $ V_N$. 
	This notation distinguishes it from the one-dimensional derivative-space projection $P_N$ used earlier.
	Define $A_N:V_N\to V_N$ and $\mathcal E_N(t)$ by
	\begin{equation*}
		(A_Nw_N,v_N)=(\grad w_N,\grad v_N),\qquad \forall v_N \in V_N; \qquad
		\mathcal E_N(t)=e^{-tA_N}.
	\end{equation*}
	The operator $A_N$ is positive definite and self-adjoint. 
	The tensor-product Babu\v{s}ka--Shen basis of \(V_N\) is defined by
	\[
	\{\eta_{k\ell}\}_{k,\ell=2}^N,\qquad
	\eta_{k\ell}(x,y):=\eta_k(x)\eta_\ell(y),\qquad (x,y)\in I^2.
	\]
	Thus every \(v_N\in V_N\) has a unique expansion $v_N(x,y)=\sum_{k,\ell=2}^N c_{k\ell}\eta_{k\ell}(x,y),$ with coefficient vector \(c=(c_{k\ell})_{k,\ell=2}^N\). With respect to this basis, $A_N$ and $\mathcal E_N(t)$ can be represented by matrices.


	Define $\zeta=u-R_Nu$ and $\theta=R_Nu-u_N$, so that $\zeta(0)=u_0-R_Nu_0$ and $\theta(0)=R_Nu_0-u_{N,0}$. Then
	\begin{equation}\label{eq:parabolic-splitting}
		u-u_N=\zeta+\theta.
	\end{equation}
	The map $R_N:H^1_0(\Om)\to H^1_0(\Om)$ is bounded. Hence the Bochner derivative satisfies
	$(R_Nu)_t=R_Nu_t$.
	Subtracting \eqref{eq:heat-galerkin} from the weak equation and using
	Ritz orthogonality gives
	\begin{equation}\label{eq:parabolic-error-equation}
		(\theta_t,v_N)+(\grad\theta,\grad v_N)=-(\zeta_t,v_N),
		\quad \forall v_N\in V_N.
	\end{equation}
	
	Throughout this section, operators on $V_N$ and their matrices in the tensor Babu\v{s}ka--Shen basis are denoted by the same symbols.
	
	\begin{lemma}\label{prop:parabolic-factorization}
		In the tensor Babu\v{s}ka--Shen basis, the matrices $A_N $ and $\mathcal E_N(t)$ can be written as
		\begin{equation}\label{eq:physical-heat-factorization}
			\begin{aligned}
				A_N&=M^{-1}\otimes I_{N-1}+I_{N-1}\otimes M^{-1},\\
				\mathcal E_N(t)&=e^{-tM^{-1}}\otimes e^{-tM^{-1}}.
			\end{aligned}
		\end{equation} 
		Here \(I_{N-1}\) denotes the \((N-1)\times(N-1)\) identity matrix. 
		Moreover, the error $\theta(t)$ satisfies
		\begin{equation}\label{eq:parabolic-duhamel}
			\theta(t)=\mathcal E_N(t)\theta(0)
			-\int_0^t\mathcal E_N(t-s)\mathcal P_N\zeta_t(s)\,\dd s.
		\end{equation}
	\end{lemma}

	\begin{proof}
		In the tensor basis \(\{\eta_i(x)\eta_j(y):2\le i,j\le N\}\), separation of the integrals in the two variables gives the mass matrix \(M\otimes M\) and the stiffness matrix \(I_{N-1}\otimes M+M\otimes I_{N-1}\). The definition of $A_N$ gives $(M\otimes M)\,A_N =(I_{N-1}\otimes M+M\otimes I_{N-1}) $.  Therefore we have
		\[
		\begin{aligned}
			A_N
			&=(M\otimes M)^{-1}\bigl(I_{N-1}\otimes M+M\otimes I_{N-1}\bigr)\\
			&=M^{-1}\otimes I_{N-1}+I_{N-1}\otimes M^{-1}.
		\end{aligned}
		\]
		Since $M^{-1}\otimes I_{N-1}$ and $I_{N-1}\otimes M^{-1}$ commute,
		\[
		\begin{aligned}
			\mathcal E_N(t)
			&=e^{-tA_N}=\bigl(e^{-tM^{-1}}\otimes I_{N-1}\bigr)\bigl(I_{N-1}\otimes e^{-tM^{-1}}\bigr)\\
			&=e^{-tM^{-1}}\otimes e^{-tM^{-1}}.
		\end{aligned}
		\]
		Using the definition of \(A_N\) and \eqref{eq:parabolic-error-equation}, we obtain
		\[
		\theta_t+A_N\theta=-\mathcal P_N\zeta_t,\qquad \theta(0)=R_Nu_0-u_{N,0}.
		\]
		Variation of constants gives
		\[
		\theta(t)=\mathcal E_N(t)\theta(0)-\int_0^t\mathcal E_N(t-s)\mathcal P_N\zeta_t(s)\,\mathrm ds.
		\]
		This completes the proof. 
	\end{proof}
	
	Lemma \ref{lem:mass-spectrum} gives the range of the eigenvalues of the mass matrix \(M\). Together with \eqref{eq:physical-heat-factorization}, we obtain  \(\sigma(A_N)\subset[c,CN^4]\).
	
	
	\subsection{A Negative-Norm Estimate for the Ritz Projection Error}
	
	\begin{lemma}\label{lem:dual-approximation}
		For $w\in H^3(\Om)\cap H^1_0(\Om)$,
		\begin{equation}\label{eq:dual-approximation}
			E_N^{1,2}(w)\lesssim N^{-2}\norm{w}_{H^3(\Om)}.
		\end{equation}
	\end{lemma}
	\begin{proof}
		Recall the one-dimensional energy projection $\Pi_N$ from Section~4.2.
		For $f\in H^1_0(I)$, we have $\Pi_Nf=JP_Nf'$.
		For the normalized Legendre polynomials $\ell_n$, the operator
		$\mathcal D h=-((1-x^2)h')'$ has eigenvalues $n(n+1)$.
		Weighted Bessel's inequality and Parseval's identity give, respectively,
		\begin{equation}\label{eq:dual-legendre-tails}
			\begin{aligned}
				\norm{(\mathrm{Id}-P_N)h}_2&\lesssim N^{-1}\norm{h'}_2,
				&&h\in H^1(I),\ \textstyle\int_Ih=0,\\
				\norm{(\mathrm{Id}-P_N)h}_2&\lesssim N^{-2}\norm{h}_{H^2(I)},
				&&h\in H^2(I),\ \textstyle\int_Ih=0.
			\end{aligned}
		\end{equation}
		Indeed, the first estimate uses
		$\sum_{n\ge1}n(n+1)|(h,\ell_n)|^2
		\le\int_I(1-x^2)|h'|^2$.
		For the second, integration by parts gives
		$(\mathcal D h,\ell_n)=n(n+1)(h,\ell_n)$ and
		$\norm{\mathcal D h}_2\lesssim \norm{h}_{H^2}$.
		The vanishing factor $1-x^2$ eliminates endpoint terms.
		
		Using $(2n+1)\sqrt{2/(2n+1)}\,\ell_n=\sqrt{2/(2n+3)}\,\ell_{n+1}'-\sqrt{2/(2n-1)}\,\ell_{n-1}'$, we obtain
		\begin{equation*}
			J\ell_n=
			\frac{\ell_{n+1}}{\sqrt{(2n+1)(2n+3)}}
			-\frac{\ell_{n-1}}{\sqrt{(2n+1)(2n-1)}},\qquad n\ge1.
		\end{equation*}
		Moreover, $\norm{Jh}_2\lesssim N^{-1}\norm{h}_2$ whenever $h$ contains
		only modes $n\ge N$.  Each output coefficient is a linear combination of at most two input coefficients.  Applying this and \eqref{eq:dual-legendre-tails} to $f'$ yields
		\begin{equation}\label{eq:dual-one-dimensional}
			\begin{aligned}
				\norm{f-\Pi_Nf}_2&\lesssim N^{-2}\norm{f}_{H^2(I)},\qquad
				\norm{(\Pi_Nf)'}_2\le\norm{f'}_2,\\
				\norm{(f-\Pi_Nf)'}_2&\lesssim N^{-2}\norm{f}_{H^3(I)}.
			\end{aligned}
		\end{equation}
		Let $P_N^x$ and $\Pi_N^x$ denote the action of $P_N$ and $\Pi_N$ in the $x$ variable, and define $P_N^y$ and $\Pi_N^y$ analogously. 
		Taking $\widetilde S_Nw=\Pi_N^x\Pi_N^yw\in V_N$, commutation of operators acting in different variables gives
		\begin{equation*}
		\partial_x(w-\widetilde S_Nw)
		=(\mathrm{Id}-P_N^x)\partial_xw+P_N^x(\mathrm{Id}-\Pi_N^y)\partial_xw.
		\end{equation*}
		The first term is bounded by the second estimate in \eqref{eq:dual-legendre-tails}. For the second term, apply the \(L^2\) estimate in \eqref{eq:dual-one-dimensional} in the \(y\) variable and use the \(L^2\) contractivity of \(P_N^x\). These require at most the derivatives $\partial_x^3w$ and $\partial_x\partial_y^2w$, together with lower derivatives. 
		On the edges $y=\pm1$, differentiation in $x$ is tangential, so the zero trace of $w$ implies that $\partial_x w$ also has zero trace there.  The same argument with
		the variables reversed controls the $y$ derivative.  This proves
		\eqref{eq:dual-approximation}.
	\end{proof}
	We define the negative Sobolev norm as the dual norm on \(H^{-1}(\Om)\) with respect to \(H^1_0(\Om)\), namely
	\begin{equation*}
		\norm{g}_{H^{-1}}=
		\sup_{0\ne v\in H^1_0(\Om)}\frac{|\langle g,v\rangle|}{\norm{\grad v}_2}.
	\end{equation*}
	\begin{lemma}\label{lem:ritz-negative}
		For every $v\in H^1_0(\Om)$,
		\begin{equation}\label{eq:ritz-negative}
			\norm{(\mathrm{Id}-R_N)v}_{H^{-1}(\Om)}
			\lesssim N^{-2}E_N^{1,2}(v).
		\end{equation}
	\end{lemma}
	\begin{proof}
		Let $\zeta=(\mathrm{Id}-R_N)v$. To estimate its $H^{-1}(\Om)$ norm, fix $g\in H_0^1(\Om)$ and let $\phi\in H_0^1(\Om)$ solve
		\[
		\begin{aligned}
			-\Delta\phi&=g\quad\text{in }\Om,\\
			\qquad \phi&=0\quad\text{on }\partial\Om.
		\end{aligned}
		\]
		We first verify the regularity needed to apply Lemma~\ref{lem:dual-approximation}. The orthonormal Dirichlet eigenfunctions on $\Om=(-1,1)^2$ and their eigenvalues are
		\[
		e_{jk}(x,y)=\sin\frac{j\pi(x+1)}2\sin\frac{k\pi(y+1)}2,\quad
		\lambda_{jk}=\frac{\pi^2(j^2+k^2)}4,\quad j,k\ge1.
		\]
		Writing $\phi_{jk}=(\phi,e_{jk})$ and $g_{jk}=(g,e_{jk})$, the weak equation gives $\lambda_{jk}\phi_{jk}=g_{jk}$. Orthogonality of the differentiated sine and cosine functions, together with $g\in H_0^1(\Om)$, yields
		\[
		\norm{\phi}_{H^3(\Om)}^2
		\lesssim \sum_{j,k\ge1}\lambda_{jk}^3|\phi_{jk}|^2
		\lesssim\sum_{j,k\ge1}\lambda_{jk}|g_{jk}|^2
		=\norm{\grad g}_2^2.
		\]
		Let $\widetilde S_N\phi\in V_N$ be the approximant constructed in Lemma~\ref{lem:dual-approximation}. The preceding estimate gives
		\[
		\norm{\grad(\phi-\widetilde S_N\phi)}_2
		\lesssim N^{-2}\norm{\phi}_{H^3(\Om)}
		\lesssim N^{-2}\norm{\grad g}_2.
		\]
		The weak equation and Ritz orthogonality imply
		\[
		(\zeta,g)=(\grad\zeta,\grad\phi),\qquad
		(\grad\zeta,\grad \widetilde S_N\phi)=0.
		\]
		Subtracting the second identity from the first and applying the approximation estimate, we obtain
		\[
		|(\zeta,g)|
		=|(\grad\zeta,\grad(\phi-\widetilde S_N\phi))|
		\le \norm{\grad\zeta}_2\norm{\grad(\phi-\widetilde S_N\phi)}_2
		\lesssim N^{-2}\norm{\grad\zeta}_2\norm{\grad g}_2.
		\]
		Taking the supremum over nonzero $g\in H_0^1(\Om)$ and using the energy best-approximation property of $R_N$, we conclude that
		\[
		\norm{\zeta}_{H^{-1}(\Om)}
		=\sup_{0\ne g\in H_0^1(\Om)}\frac{|(\zeta,g)|}{\norm{\grad g}_2}
		\lesssim N^{-2}\norm{\grad\zeta}_2
		=N^{-2}E_N^{1,2}(v).
		\]
		This completes the proof.
	\end{proof}
	
	\subsection{A Hilbert-Space Convolution Bound}
	\begin{lemma}\label{lem:hilbert-convolution}
		Let $A$ be a positive self-adjoint operator on a finite-dimensional Hilbert space $H$, with spectrum in $[a,b]$, $0<a\le b$. Denote by $\|\cdot\|_H$ the norm on $H$. For $F\in L^\infty(0,T;H)$,
		\begin{equation}\label{eq:hilbert-convolution}
			\sup_{0\le t\le T}
			\left\|\int_0^t Ae^{-(t-s)A}F(s)\,\dd s\right\|_H
			\lesssim \sqrt{1+\log(b/a)}\norm{F}_{L^\infty(0,T;H)}.
		\end{equation}
	\end{lemma}
	\begin{proof}
		Fix $t\in[0,T]$ and $h\in H$ with $\|h\|_H=1$. Self-adjointness and the change of variables $r=t-s$ give
		\[
		\begin{aligned}
			\left|\left\langle\int_0^t Ae^{-(t-s)A}F(s)\,\dd s,h\right\rangle_H\right|
			&=\left|\int_0^t\langle F(s),Ae^{-(t-s)A}h\rangle_H\,\dd s\right|\\
			&\le \|F\|_{L^\infty(0,T;H)}\int_0^\infty\|Ae^{-rA}h\|_H\,\dd r.
		\end{aligned}
		\]
		To estimate the last integral, write 
		\[
		\int_0^\infty\|Ae^{-rA}h\|_H\,\dd r=\left(\int_0^{b^{-1}} + \int_{b^{-1}}^{a^{-1}}+ \int_{a^{-1}}^\infty \right) \|Ae^{-rA}h\|_H\,\dd r=
		I_1+I_2+I_3.
		\]
		The spectral theorem gives $\|Ae^{-rA}h\|_H\le\max_{\lambda\in[a,b]}\lambda e^{-r\lambda}$. For $r\ge a^{-1}$, the function $\lambda\mapsto\lambda e^{-r\lambda}$ is decreasing on $[a,b]$. Hence
		\[
		I_1\le\int_0^{b^{-1}}b\,\dd r=1,\qquad
		I_3\le\int_{a^{-1}}^\infty ae^{-ar}\,\dd r=e^{-1}.
		\]
		For $I_2$, weighted Cauchy--Schwarz gives
		\begin{align*}
				I_2&\le
		\left(\int_{b^{-1}}^{a^{-1}}\frac{\dd r}{r}\right)^{1/2}
		\left(\int_0^\infty r\|Ae^{-rA}h\|_H^2\,\dd r\right)^{1/2}\\
		&=\sqrt{\log(b/a)}
		\left(\int_0^\infty r\|Ae^{-rA}h\|_H^2\,\dd r\right)^{1/2}.
		\end{align*}
		To evaluate the second factor, choose an orthonormal eigenbasis $\{e_j\}_{j=1}^{\dim H}$ of $A$, with $Ae_j=\lambda_je_j$ and $\lambda_j\in[a,b]$. Then
		\[
		\int_0^\infty r\|Ae^{-rA}h\|_H^2\,\dd r
		=\sum_{j=1}^{\dim H}|\langle h,e_j\rangle_H|^2
		\int_0^\infty r\lambda_j^2e^{-2r\lambda_j}\,\dd r
		=\frac14\|h\|_H^2=\frac14.
		\]
		Combining the three estimates yields
		\[
		I_1+I_2+I_3\le 1+e^{-1}+\frac12\sqrt{\log(b/a)}
		\lesssim \sqrt{1+\log(b/a)}.
		\]
		Substituting this bound into the duality estimate and taking the supremum over $\|h\|_H=1$ and $t\in[0,T]$ proves \eqref{eq:hilbert-convolution}, completing the proof.
	\end{proof}
	
	\subsection{The Algebraically Sharp Error Estimate}	
	\begin{theorem}
		\label{thm:parabolic-error}
		Let $u\in C([0,T];W)\cap W^{1,\infty}(0,T;H_0^1(\Omega))$
		be a weak solution of the parabolic problem. Let $u_N$ solve
		\eqref{eq:heat-galerkin} with the Ritz initial value
		$u_{N,0}=R_Nu_0=R_Nu(0)$. Then
		\begin{equation}\label{eq:parabolic-main}
			\begin{aligned}
				\sup_{0\le t\le T}\norm{u(t)-u_N(t)}_{W^{1,\infty}}
				&\lesssim \sqrt N \Lambda_N^3\sup_{0\le t\le T}E_N^{1,\infty}(u(t))\\
				&\quad+\sqrt{\Lambda_N}\,
				\esssup_{0<t<T}E_N^{1,2}(u_t(t)).
			\end{aligned}
		\end{equation}
	\end{theorem}
	\begin{proof}
		Recall that $u-u_N=\zeta+\theta$, where $\zeta=u-R_Nu$ and $\theta=R_Nu-u_N$. The elliptic stability estimate controls $\zeta$. To estimate $\theta$, we combine the negative-norm estimate, the Hilbert-space convolution bound, and the polynomial inverse inequality.
		By \eqref{eq:ritz-best-infty},
		\[
		\sup_{0\le t\le T}\norm{\grad\zeta(t)}_\infty
		\lesssim \sqrt{N}\,\Lambda_N^3\sup_{0\le t\le T}E_N^{1,\infty}(u(t)).
		\]
		To combine the negative-norm estimate with the convolution bound, define
		\[
		G_N(t):=A_N^{-1/2}\mathcal P_N\zeta_t(t)\in V_N.
		\]
		Its $L^2$ norm represents the discrete negative norm of the projected error derivative $\mathcal P_N\zeta_t(t)$.
		
		Since $A_N$ is self-adjoint and $\mathcal P_N$ is the $L^2(\Om)$-orthogonal projection onto $V_N$, we have
		\begin{equation*}
			(G_N(t),A_N^{1/2}v_N) =(A_N^{-1/2}\mathcal P_N\zeta_t(t) ,A_N^{1/2}v_N)	=(\zeta_t(t),v_N). 
		\end{equation*}
		Since $A_N^{1/2}$ maps $V_N$ bijectively onto itself, $L^2$ duality and Lemma~\ref{lem:ritz-negative} yield
		\begin{equation}\label{eq:discrete-negative-residual}
			\begin{aligned}
				\norm{G_N(t)}_2
				&=\sup_{0\ne v_N\in V_N}\frac{|(\zeta_t(t),v_N)|}{\norm{\grad v_N}_2}\\
				&\le \norm{\zeta_t(t)}_{H^{-1}(\Omega)}
				\lesssim N^{-2}E_N^{1,2}(u_t(t)).
			\end{aligned}
		\end{equation}		
		Since $\theta(0)=0$ and spectral functions of $A_N$ commute, \eqref{eq:parabolic-duhamel} gives
		\[
		\theta(t)=-\int_0^t\mathcal E_N(t-s)\mathcal P_N\zeta_t(s)\,\dd s,\quad
		A_N^{1/2}\theta(t)=-\int_0^t A_N\mathcal E_N(t-s)G_N(s)\,\dd s.
		\]
		The spectral bound of $ A_N$ implies
		\[
		1+\log\frac{\lambda_{\max}(A_N)}{\lambda_{\min}(A_N)}
		\le 1+\log(CN^4/c)\lesssim \Lambda_N.
		\]
		Applying Lemma~\ref{lem:hilbert-convolution} and then \eqref{eq:discrete-negative-residual}, we obtain
		\[
		\begin{aligned}
			\sup_{0\le t\le T}\norm{\grad\theta(t)}_2
			&\lesssim \sqrt{\Lambda_N}\operatorname*{ess\,sup}_{0<t<T}\norm{G_N(t)}_2\\
			&\lesssim N^{-2}\sqrt{\Lambda_N}\operatorname*{ess\,sup}_{0<t<T}E_N^{1,2}(u_t(t)).
		\end{aligned}
		\]
		Lemma~\ref{lem:gradient-inverse} then yields
		\[
		\sup_{0\le t\le T}\norm{\grad\theta(t)}_\infty
		\lesssim \sqrt{\Lambda_N}\operatorname*{ess\,sup}_{0<t<T}E_N^{1,2}(u_t(t)).
		\]
		Since $\zeta$ and $\theta$ have zero trace, Poincar\'e's inequality gives
		\[
		\begin{aligned}
			&\sup_{0\le t\le T}\norm{u(t)-u_N(t)}_{W^{1,\infty}(\Omega)}\\
			&\lesssim \sup_{0\le t\le T}\bigl(\norm{\grad\zeta(t)}_\infty+\norm{\grad\theta(t)}_\infty\bigr)\\
			&\lesssim \sqrt{N}\,\Lambda_N^3\sup_{0\le t\le T}E_N^{1,\infty}(u(t))
			+\sqrt{\Lambda_N}\operatorname*{ess\,sup}_{0<t<T}E_N^{1,2}(u_t(t)).
		\end{aligned}
		\]
		This proves \eqref{eq:parabolic-main}.
	\end{proof}
	
	The following proposition shows that the spatial factor in \eqref{eq:parabolic-main} cannot be replaced by $o(\sqrt N)$, establishing the optimality of its algebraic exponent. 
	\begin{proposition}\label{prop:parabolic-sharpness}
		For every $N\ge2$, there exists a smooth stationary solution satisfying 
		\eqref{eq:parabolic-regularity} such that
		\begin{equation*}
			\sqrt N\lesssim1+\sup_t\norm{\grad(u(t)-u_N(t))}_\infty,
			\qquad \sup_t E_N^{1,\infty}(u(t))\lesssim 1.
		\end{equation*}
	\end{proposition}
	\begin{proof}
		We first choose smooth test functions that retain the lower bound in Lemma~\ref{prop:linfty-lower}. Let $k_N$ and $f_N$ be the projection kernel and mean-zero test function from Lemma~\ref{lem:epd-endpointlower}. For $m\ge1$, define
		\[
		f_{N,m}(x):=\tanh(mk_N(1,x))-\frac12\int_I\tanh(mk_N(1,s))\,\dd s.
		\]
		Then $\int_I f_{N,m}=0$ and $\|f_{N,m}\|_\infty\le2$. For each fixed $N$, dominated convergence gives $f_{N,m}\to f_N$ in $L^2(I)$ as $m\to\infty$.
		Using the integration operator $J$ and $\eta_2(y)=\sqrt{3/8}(1-y^2)$ from the lower-bound construction, set
		\[
		w_{N,m}(x,y):=(Jf_{N,m})(x)\eta_2(y).
		\]
		Since $Jf_{N,m}$ vanishes at both endpoints, we have
		\[
		w_{N,m}\in C^\infty(\overline\Om)\cap H^1_0(\Om),\qquad \norm{\grad w_{N,m}}_\infty\lesssim 1.
		\]
		For each fixed $N$, boundedness of $R_N$ in $H^1_0(\Om)$ and equivalence of norms on $V_N$ give
		\[
		\norm{\grad R_N\bigl(w_{N,m}-(Jf_N)\eta_2\bigr)}_\infty\longrightarrow0.
		\]
		The lower bound from Lemma~\ref{prop:linfty-lower} therefore allows us to choose $m$ sufficiently large and write $w_N=w_{N,m}$, satisfying
		\[
		\norm{\grad w_N}_\infty\lesssim 1,\qquad \norm{\grad R_Nw_N}_\infty\ge c\sqrt N.
		\]
		
		Set $u(t)=w_N$, $u_0=w_N$, and $f=-\Delta w_N$. Then $u$ is a stationary solution of \eqref{eq:parabolic-eq} satisfying \eqref{eq:parabolic-regularity}, with $u_t=0$. Initialize the Galerkin approximation by $u_N(0)=R_Nw_N$. The defining identity of $R_N$ gives
		\[
		(\grad R_Nw_N,\grad v_N)=(\grad w_N,\grad v_N)=(f,v_N),\qquad v_N\in V_N.
		\]
		Consequently, $R_Nw_N$ solves the semidiscrete problem. The uniqueness therefore yields $u_N(t)=R_Nw_N$. It follows that
		\[
		\sup_{0\le t\le T}\norm{\grad(u(t)-u_N(t))}_\infty
		\ge \norm{\grad R_Nw_N}_\infty-\norm{\grad w_N}_\infty.
		\]
		The lower bounds for $\norm{\grad R_Nw_N}_\infty$ and $\norm{\grad w_N}_\infty\lesssim1$ imply $\sqrt N\lesssim1+\sup_t\norm{\grad(u(t)-u_N(t))}_\infty$. Taking $v_N=0$ in the definition of $E_N^{1,\infty}$ gives
		\[
		\begin{aligned}
			\sup_{0\le t\le T}E_N^{1,\infty}(u(t))
			&=E_N^{1,\infty}(w_N)\\
			&=\inf_{v_N\in V_N}\|\nabla(w_N-v_N)\|_\infty\\
			&\le\|\nabla(w_N-0)\|_\infty =\|\nabla w_N\|_\infty\lesssim 1.
		\end{aligned}	
		\]
		This completes the proof.
	\end{proof}
	
\section*{AI Disclosure}
During the preparation of this work, GPT-5.6 sol and GPT-6.0 Astra (OpenAI) were used to explore proof strategies and assist mathematical derivations. The authors reviewed and revised the AI-assisted material and independently verified all mathematical statements and proofs. The authors assume full responsibility for the results and the final manuscript.
	
\bibliographystyle{amsplain}
\bibliography{ref}
		
\end{document}